\documentclass[12pt,psfig,reqno]{amsart}
\usepackage{}
\usepackage{mathrsfs}
\usepackage{txfonts}
\usepackage{cite}
\usepackage{amsmath}
\usepackage{bbm}
\usepackage{amssymb}
\usepackage{amsfonts}
\usepackage{amscd}
\usepackage{amssymb,amsfonts,latexsym,enumerate}
\usepackage{latexsym,amsmath,amssymb,amsthm,amsfonts,arydshln,enumerate,graphicx}
\usepackage{graphics,verbatim}
\usepackage{graphicx}
\usepackage{chngcntr}
\usepackage{pifont}
\counterwithin{figure}{section}

\usepackage[usenames]{color} 
\usepackage[usenames]{color} 
\usepackage[colorlinks,
            linkcolor=cyan,
            anchorcolor=cyan,
            citecolor=cyan
            ]{hyperref}
\usepackage{tikz}
\usepackage{multirow}
\usepackage{subfigure}
\usepackage{diagbox}

\usetikzlibrary{shapes,snakes}
\allowdisplaybreaks
\makeatletter
\@namedef{subjclassname@2020}{\textup{2020} Mathematics Subject Classification}
\makeatother

\makeatletter

\renewcommand\section{\@startsection {section}{1}{\z@}%
                                   {-3.5ex \@plus -1ex \@minus -.2ex}%
                                   {2.3ex \@plus.2ex}%
                                   {\centering\normalfont\bf}}

 \numberwithin{equation}{section}
\numberwithin{equation}{section}

\numberwithin{equation}{section}

\theoremstyle{plain}
\newtheorem{thm}{Theorem}[section]
\newtheorem{lem}[thm]{Lemma}
\newtheorem{pro}[thm]{Proposition}

\newtheorem{ex}[thm]{Example}
\newtheorem{definition}[thm]{Definition}
\newtheorem{re}[thm]{Remark}

\newtheorem{claim*}{Claim}

\newtheorem*{question*}{(Qu)}

\newtheorem*{thm*}{Theorem}

\begin{document}
\baselineskip=17pt
\title[Spectral measures with $m$-alternate contraction ratios ]{A class of spectral measures with $m$-alternate contraction ratios in $\mathbb{R}$ }
\author{Jing-cheng Liu, Jia-jie Wang$^{*}$}

\address{Key Laboratory of Computing and Stochastic Mathematics (Ministry of Education), School of Mathematics and Statistics, Hunan Normal University, Changsha, Hunan 410081,  P.R. China}
\email{jcliu@hunnu.edu.cn}
\address{Key Laboratory of Computing and Stochastic Mathematics (Ministry of Education), School of Mathematics and Statistics, Hunan Normal University, Changsha, Hunan 410081,  P.R. China}
\email{wjj2021hnsd@163.com}
\date{\today}
\thanks{
The research is supported in part by the NNSF of China (No.12071125), the Hunan Provincial NSF (No.2024JJ3023) and the education Department Important Foundation of Hunan province in China (No.23A0059).\\
*Corresponding author}
\begin{abstract}
 For a Borel probability measure $\mu$ on $\mathbb{R}^{n}$, it is called a spectral measure if the Hilbert space $L^{2}(\mu)$ admits an orthogonal basis of exponential functions. In this paper, we study the spectrality of fractal measures generated by an iterated function system (IFS) with
$m$-periodic alternating contraction ratios. Specifically, for fixed $m,N\in\mathbb{N}^{+}$ and  $\rho\in(0,1)$, we define
the IFS as follows:
$$\{\tau_d(\cdot)=(-1)^{\lfloor\frac{d}{m}\rfloor}\rho(\cdot+d)\}_{d\in D_{2Nm}},$$
where $D_k=\{0,1,\cdots,k-1\}$ and $\lfloor x\rfloor$ denotes the floor function. We prove that the associated self-similar measure $\nu_{\rho,D_{2Nm}}$  is a spectral measure if and only if $\rho^{-1}=p\in\mathbb{N}$ and $2Nm\mid p$.
Furthermore, for any positive integers $p,s\geq2$, if $m=1$ and $\gcd(p,s)=1$ we show that $\nu_{p^{-1},D_{s}}$ is not a spectral measure and $L^2(\nu_{p^{-1},D_{s}})$ contains at most $s$ mutually orthogonal exponential functions. These results generalize recent work of Wu \cite{Wu2024}[H.H. Wu, Spectral self-similar measures with alternate contraction ratios and consecutive digits, {\it Adv. Math.}, {\bf443} (2024), 109585]. 
\end{abstract}

\maketitle
\section{Introduction}

Let $D=\{d_0,d_1,\cdots,d_{N-1}\}\subset\mathbb{R}^{n}$ be a finite digit set with the cardinality $\#D=N$, and let $\{M_k\}_{k=0}^{N-1}$ be a sequence of expanding matrices of $M_n(\mathbb{R})$. The \textit{iterated function system} (IFS) is defined as
\begin{equation}\label{1.1}
\tau_k(x)=M_k^{-1}(x+d_k), \quad  x\in\mathbb{R}^{n}, \quad  0\leq k\leq N-1.
\end{equation}
By Hutchinson \cite{Hut1981},  there exists a unique probability measure $\mu$ satisfying
\begin{equation}\label{1.2}
\mu(\cdot)=\frac{1}{N}\sum_{k=0}^{N-1}\mu\circ \tau_{k}^{-1}(\cdot).
\end{equation}
This measure $\mu$ is supported on a unique nonempty compact set
$T=\bigcup_{k=0}^{N-1}\tau_{k}(T)$, and is known as a \textit{self-affine measure}. In particular, when each matrix $M_k$ is a scalar multiple of an orthonormal matrix, $\mu$ is called a \textit{self-similar measure}. For simplicity, when all matrices $M_k$ in (\ref{1.1}) are identical (i.e., $M_k=M$), we denote the corresponding measure by $\mu_{M^{-1}, D}$.

In this paper, we primarily investigate the spectrality of a class of self-similar measures in $\mathbb{R}$.
We say that a measure $\mu$ is a {\it spectral measure} if the Hilbert space $L^{2}(\mu)$ contains an orthogonal basis consisting of exponential functions $\{e^{2\pi i\langle\lambda,x\rangle}\}_{\lambda\in\Lambda}$. In such case,
$\Lambda$ is referred to as a {\it spectrum} of $\mu$. In harmonic analysis, the investigation of spectral properties of Borel measures has emerged as a central research direction, particularly in understanding their Fourier-analytic characteristics.  If $\mu$ is a spectral measure, then every function $f \in L^2(\mu)$ can be represented through a nontrivial Fourier series expansion with respect to an orthogonal basis of exponential functions.
The interest in researching the spectrality of self-affine measures can be traced back to the seminal work of Jorgensen and Pedersen \cite{JP98} in 1998. They demonstrated that the middle-fourth Cantor measure is a spectral measure, whereas the middle-third Cantor measure is nonspectral. This discovery spurred extensive research into the spectrality of singular measures. In $\mathbb{R}$, one influential class of measures   is the Bernoulli convolution measure $\mu_{\rho,\{0,1\}}$ with $0<\rho<1$. Their spectrality was initially investigated by Hu and Lau \cite{Hu-Lau2008}, and later, Dai {\it et al.} \cite{Dai2012,Dai-He-Lau2014} provided a complete characterization, extending the results to the more general
$N$-Bernoulli convolution measures $\mu_{\rho,\{0,1,\cdots, N-1\}}$.
In higher dimensions, several interesting examples of spectral self-affine measures have been studied, including Sierpinski-type measures \cite{Dai-Fu-Yan2021,Lu-Dong-Liu2022} and Cantor-dust-type measures \cite{Chen-Liu-Wang2023}.

Given a finite set $E\subset\mathbb{R}^{n}$, we define the uniform discrete measure $\delta_{E}=\frac{1}{\#E}\sum_{e\in E}\delta_{e}$, where $\delta_{e}$ denotes the Dirac measure at the point $e\in E$. The self-affine measure $\mu_{B^{-1}, R}$ is well known to admit a representation as an infinite convolution of discrete measures:
$$\mu_{B^{-1}, R}=\delta_{B^{-1}R}*\delta_{B^{-2}R}*\cdots*\delta_{B^{-n}R}*\cdots.$$
If the matrix $B$ and the digit set $R$ are allowed to vary at each iteration, we obtain a more general infinite convolution measure
\begin{align}\label{1.5}
\delta_{B_{1}^{-1}R_1}*\delta_{B_{1}^{-1}B_{2}^{-1}R_2}*\cdots*\delta_{B_{1}^{-1}B_{2}^{-1}\cdots B_{n}^{-1}R_n}*\cdots,
\end{align}
where $\{B_i\}_{i=1}^{\infty}$ is a sequence of invertible matrices in $\mathbb{R}^{n}$ and $\{R_i\}_{i=1}^{\infty}\subset\mathbb{R}^{n}$ is a sequence of digit sets. If the infinite convolution in (\ref{1.5}) converges in a weak sense, we denote the weak limit by $\mu_{\{B_n\},\{R_n\}}$ and refer to it as a {\it Moran measure}.
The spectrality of Moran measures enrich our understanding of spectral measures, and remain an active area of research \cite{Strichartz2000, An-He2014,Deng-Li_2022,He-He2017,Fu-Tang2024,Luo-Mao-Liu2025,Li-Wu2024,Li-Miao-Wang2024}.

To date, the majority of work on self-affine or self-similar measures has been limited to the case of uniform contraction ratios, i.e., $M_k=M$ for all $k=0,1,\cdots,N-1$ in (\ref{1.1}). Recently, Wu and Liu \cite{Wu-Liu2022} constructed a class of spectral measures $\nu_{\rho,D}$ generated by the iterated function system 
\begin{align}\label{1.16}
\{\tau_{d_i}(\cdot)=(-1)^{i}\rho(\cdot+d_i)\}_{d_i\in D},
\end{align}
where $|\rho|\in(0,1)$ and $D\subset\mathbb{Z}$ is a finite subset. When $\#D=2$, a sufficient and necessary condition for $\nu_{\rho,D}$ to be a spectral measure was given. For convenience, we denote $D_n$ as a consecutive digit set $\{0,1,\cdots,n-1\}$, where $n\in\mathbb{N}^{+}$. Obviously, if the digit set $D=\{d_0,d_1,\cdots,d_{n-1}\}$ in (\ref{1.16}) is a consecutive digit set with $\#D=n$, then $d_i=i$ for $0\leq i\leq n-1$ and (\ref{1.16}) can be expressed as
\begin{align}\label{1.12}
\{\tau_{d_i}(\cdot)=(-1)^{i}\rho(\cdot+i)\}_{d_i\in D}.
\end{align}
Recently, Wu \cite{Wu2024} further characterise the spectrality of the self-similar measure generated by the IFS in (\ref{1.12}), associated with the consecutive digit set $D$ satisfying $\#D\in 2\mathbb{N}^{+}$.

\begin{thm}\label{thm1.0}(\cite{Wu2024})
Let $\nu_{\rho,D_{2N}}$ be the self-similar measure generated by the IFS in (\ref{1.12}), associated with a contraction ratio $0<\rho<1$ and a consecutive digit set $D_{2N}$  ($N\in\mathbb{N}^{+}$). Then $\nu_{\rho,D_{2N}}$ is a spectral measure if and only if $\rho^{-1}\in\mathbb{N}$ and $2N\mid \rho^{-1}$.
\end{thm}

Inspired by the above results, some natural questions are as follows. For fixed $m\in\mathbb{N}^{+}$, let the self-similar measure $\nu_{\rho, D}^{m}$ be generated by the IFS
\begin{align}\label{1.3}
\{\tau_{d}(\cdot)=(-1)^{\lfloor\frac{d}{m}\rfloor}\rho(\cdot+d)\}_{d\in D},
\end{align}
where $\lfloor x\rfloor$ denotes the floor function, $\rho\in(0,1)$ is a real number and $0\in D$ is a finite consecutive digit set with $\#D\in m\mathbb{N}^{+}$. Clearly, (\ref{1.3}) can degenerate to (\ref{1.12}) when $m = 1$.

\begin{question*}
What is the sufficient and necessary condition for $\nu_{\rho, D}^{m}$ to be a spectral measure? Moreover, for the case $m=1$, how to characterise the spectrality of the self-similar measure $\nu_{\rho, D_{2N+1}}^{1}$ with $N\in\mathbb{N}^+$ ?
\end{question*}

In this paper, a detailed investigation has been carried out to address the above questions and some positive results have been obtained. The following theorem is our first result.
\begin{thm}\label{thm1.1}
Let $\nu_{\rho,D_{2Nm}}^{m}$ be a self-similar measure generated by the IFS in (\ref{1.3}). Then $\nu_{\rho,D_{2Nm}}^{m}$ is a spectral measure if and only if $\rho^{-1}\in\mathbb{N}$ with $2Nm\mid \rho^{-1}$.
\end{thm}

It is worth noting that Theorem \ref{thm1.1} clearly reduces to Theorem \ref{thm1.0} by taking $m=1$. Our key insight involves transforming the spectral analysis of $\nu_{\rho,D_{2Nm}}^{m}$ into that of a measure generated by an IFS with uniform contraction ratios. Specifically, through an ingenious transformation (distinct from the methods in \cite{Wu2024} and detailed in Proposition \ref{lem3}), we demonstrate that $\nu_{\rho,D_{2Nm}}^{m}=\mu_{\rho,D}$, where
\begin{align}\label{1.11}
D=D_m\oplus 2mD_N \oplus (1+m\rho-2Nm)D_2,
\end{align}
and the corresponding IFS of $\mu_{\rho,D}$ is
\begin{align}\label{1.10}
\{\tau_d(\cdot)=\rho(\cdot+d)\}_{d\in D}.
\end{align}
In this way, we only need to analyze the spectrality of $\mu_{\rho,D}$, from which the spectrality of $\nu_{\rho,D_{2Nm}}^{m}$ is immediately available. Thus, Theorem \ref{thm1.1} can be obtained directly from the following theorem.
\begin{thm}\label{thm1.2}
Let $\mu_{\rho,D}$ be a self-similar measure generated by the IFS in (\ref{1.10}), associated with  $0<\rho<1$ and the digit set $D$ defined by (\ref{1.11}).
Then the following statements are equivalent.
\begin{enumerate}[(i)]
 \item $\mu_{\rho,D}$ is a spectral measure;

 \item $\rho^{-1}=p\in\mathbb{N}$ and $2Nm\mid p$;

 \item There exists $L\subset\mathbb{Z}$ such that $(p,pD,L)$ is a (2-stage) product-form Hadamard triple (see Definition \ref{Higher-stage-product-form}).
\end{enumerate}
\end{thm}

In contrast, the proof of the implication $(i)\Rightarrow(ii)$ requires the most delicate treatment, and will be presented in four key steps (detailed in Section \ref{sec3}). Among these, the step (D) ( i.e., showing $2Nm \mid p$ when $\rho^{-1}=p\in\mathbb{N}$ ) presents the greatest challenge. Due to the arbitrariness of $m$ and $N$, the proof needs to be discussed separately based on their parity.
Especially for the case \(m \in 2\mathbb{N}\) and \(N \in 2\mathbb{N}+1\) (a scenario not addressed in \cite{Wu2024}), we develop a new approach to resolve this complication. Drawing upon the existing work on Moran measures, we reformulate the self-similar measures in Moran measure terms, thereby enabling us to apply established methods and conclusions to obtain our result. This same idea also applies to the step (C).

For the measure $\nu_{\rho,D_{(2N+1)m}}^{m}$, it remains extremely challenging to devise an efficient method for transforming $\nu_{\rho,D_{(2N+1)m}}^{m}$ into a self-similar measure with the uniform contraction ratio. Consequently, the spectrality of $\nu_{\rho,D_{(2N+1)m}}^{m}$ has remained intractable.  However, for the measure $\nu_{\rho, D_s}^{1}$  with $\rho^{-1},s\in\mathbb{N}^{+}$, we can characterize its nonspectral properties as follows.

\begin{thm}\label{thm1.4}
Let $\nu_{\rho,D_s}^{1}$ be a self-similar measure generated by the IFS in (\ref{1.12}), where $\rho^{-1}=p\in\mathbb{N}$ with $p\geq 2$ and $s\geq 2$. Assume that $\gcd (p,s)=1$, then $\nu_{\rho,D_s}^{1}$ is not a spectral measure. In particular, $L^2(\nu_{\rho,D_{s}}^{1})$ contains at most $s$ mutually orthogonal exponential functions.
\end{thm}

The paper is organized as follows. In Section \ref{sec2}, we first introduce some fundamental concepts and results which will be used throughout the paper. Then we show that the measure $\nu_{\rho,D_{2Nm}}^{m}$ is equal to the measure $\mu_{\rho,D}$, where $D$ is defined by (\ref{1.11}). In Section \ref{sec3}, we investigate the spectrality of $\mu_{\rho,D}$ and give a complete proof of Theorem \ref{thm1.2}.
In Section \ref{sec4}, we study the nonspectral properties of $\nu_{\rho,D_{s}}^{1}$ associated with the IFS defined by (\ref{1.12}), and prove Theorem \ref{thm1.4}.

\section{Preliminaries\label{sec2}}
In this section, we introduce some basic definitions and useful conclusions related to self-similar measures and Moran measures in $\mathbb{R}$. Since the Moran measure is a generalisation of the self-similar measure, we will present some notations in a more general form.

Let $\{b_n\}_{n=1}^{\infty}$ be a sequence of real numbers and $\{R_n\}_{n=1}^{\infty}$ be a sequence of digits sets with $0\in R_n\subset\mathbb{R}$.
Suppose that the Moran measure $\mu_{\{b_n\},\{R_n\}}$ exists. Then it can be written in the form of the following infinite convolution of discrete measures
\begin{align}\label{1.4}
\mu_{\{b_n\},\{R_n\}}=\delta_{b_{1}^{-1}R_1}*\delta_{b_{1}^{-1}b_{2}^{-1}R_2}*\cdots*\delta_{b_{1}^{-1}b_{2}^{-1}\cdots b_{n}^{-1}R_n}*\delta_{b_{1}^{-1}b_{2}^{-1}\cdots b_{n+1}^{-1}R_{n+1}}*\cdots.
\end{align}
It is widely known that if the infinite convolution in (\ref{1.4}) satisfies
$$\sum_{n=1}^\infty\frac{\max\{|a|:a\in R_n\}}{|b_1b_2\cdots b_n|}<\infty,$$
then $\mu_{\{b_n\},\{R_n\}}$ is a Borel probability measure with compact support \cite{Strichartz2000}.  Recently, Li {\it et al.} \cite{Li-Miao-Wang2022} have  established a complete characterization for the existence of the infinite convolution in (\ref{1.4}) under the assumption that $b_1^{-1}b_2^{-1}\cdots b_k^{-1}R_k$ is a subset of $[0,+\infty)$ for each $k\geq 1$.
For the convenience of the subsequent discussion, we define a measure related to $\mu_{\{b_n\},\{R_n\}}$ as follows:
\begin{equation}\label{1.15}
\mu_{\{b_n,R_n,>n\}}:=\delta_{b_{n+1}^{-1}R_{n+1}}*\delta_{b_{n+1}^{-1}b_{n+2}^{-1}R_{n+2}}*\delta_{b_{n+1}^{-1}b_{n+2}^{-1}b_{n+3}^{-1}R_{n+3}}*\cdots
\end{equation}
for $n\geq 1$.

For a compactly supported Borel probability measure $\mu$ on $\mathbb{R}$, the Fourier transform of $\mu$ is defined by
\begin{align}\label{1.13}
\hat{\mu}(\xi)=\int_{\mathbb{R}}e^{2\pi i\langle x,\xi\rangle}\,\textup{d}\mu(x),\quad \xi\in\mathbb{R}.
\end{align}
By a direct calculation, the Fourier transform of the Moran measure $\mu_{\{b_n\},\{R_n\}}$ is
\begin{equation}\label{1.7}
\hat{\mu}_{\{b_n\},\{R_n\}}(\xi)=\prod\limits_{j=1}^{\infty}m_{R_j}((b_{1}b_{2}\cdots b_{j})^{-1}\xi),
\end{equation}
where $m_{R}(\cdot)$ is the mask polynomial of the digit set $R$, i.e.,
\begin{align*}
m_{R}(x)=\frac{1}{\#R}\sum_{d\in R}e^{2\pi i\langle d,x\rangle},\quad x\in \mathbb{R}.
\end{align*}

Let $\mathcal{Z}(f)=\{x:f(x)=0\}$ denote the zero set of a function $f$, then it directly follows from \eqref{1.7} that
\begin{equation}\label{1.8}
\mathcal{Z}(\hat{\mu}_{\{b_n\},\{R_n\}})=\bigcup_{n=1}^{\infty}b_{1}b_{2}\cdots b_{n}\mathcal{Z}(m_{R_n}).
\end{equation}
In the special case where $b_n\equiv b>1$ and $R_n\equiv R$ for all $n$, the measure $\mu_{\{b_n\},\{R_n\}}$ reduces to the canonical self-similar measure $\mu_{\rho,R}$ generated by
\begin{align}\label{1.6}
\mu_{\rho,R}(\cdot)=\frac{1}{\#R}\sum_{d\in R}\mu_{\rho,R}(\rho^{-1}(\cdot)-d),
\end{align}
where $\rho=b^{-1}$ is a contraction ratio in $(0,1)$.

It is easy to verify that $\Lambda$ is an orthonormal set of $\mu$
if and only if $\hat{\mu}(\lambda_1-\lambda_2)=0$ holds for any $\lambda_1\neq\lambda_2\in \Lambda$, i.e.,
\begin{equation}\label{1.9}
 (\Lambda-\Lambda)\setminus\{0\}\subset\mathcal{Z}(\hat{\mu}).
\end{equation}
A basic criterion for an orthogonal set to be a spectrum of a measure is given by Jorgensen and Pedersen in \cite{JP98}.
\begin{lem}[\cite{JP98}]\label{pro6}
Let $\mu$ be a Borel probability measure with compact support on $\mathbb{R}^n$, and let $Q_{\mu,\Lambda}(\cdot)=\sum_{\lambda\in\Lambda}|\hat{\mu}(\cdot+\lambda)|^2$
for a countable set $\Lambda\subset\mathbb{R}^n$. Then
\begin{enumerate}[(i)]
\item $\Lambda$ is an orthogonal set of $\mu$ if and only if $Q_{\mu,\Lambda}(\xi)\leq1$ for all $\xi\in\mathbb{R}^n$;
\item $\Lambda$ is a spectrum of $\mu$ if and only if $Q_{\mu,\Lambda}(\xi)\equiv1$ for all $\xi\in\mathbb{R}^n$.
\end{enumerate}
\end{lem}

By applying Lemma \ref{pro6}, Dai {\it et al.} \cite{Dai-He-Lau2014} gave the following result, which will be helpful in determining whether a measure is nonspectral.
\begin{thm}[\cite{Dai-He-Lau2014}]\label{th1}
Let $\mu=\mu_{1}\ast\mu_{2}$ be the convolution of two probability measures $\mu_{1}$ and $\mu_{2}$, neither of which is a Dirac measure. Suppose that $\Lambda$ is an orthogonal set of $\mu_{1}$. Then $\Lambda$ is also an orthogonal set of $\mu$, but it cannot be a spectrum of $\mu$.
\end{thm}

As an immediate consequence of Theorem \ref{th1}, we obtain the following result, which will be used several times in proving our main results.
\begin{re}\label{re2.3}
Let $\mu_{\rho,D}$ be a self-similar measure defined by (\ref{1.6}). Suppose that
the digit set $D$ admits a decomposition $D=D_1\oplus D_2\oplus\cdots \oplus D_n$, and there exist $1\leq j_1, j_2\leq n$ along with integers $m_1,m_2\geq 1$ such that $(j_1,m_1)\neq(j_2,m_2)$ and
$$
\mathcal{Z}(\hat{\delta}_{\rho^{m_1}D_{j_1}})\subset\mathcal{Z}(\hat{\delta}_{\rho^{m_2}D_{j_2}}).
$$
Then we conclude that $\mu_{\rho,D}$ is not a spectral measure by Theorem \ref{th1}.
\end{re}

By using Theorem \ref{th1} and Ramsey's Theorem, An and Wang \cite{An-Wang2021} proved the following result related to the infinite orthogonal set of measures.
\begin{lem}[\cite{An-Wang2021}]\label{lem1}
Let $\mu=\mu_1\ast\mu_2$ be the convolution of two probablity measures $\mu_{1}$ and $\mu_{2}$. Then $\mu$ has an infinite orthogonal set if and only if some $\mu_{i}$ has an infinite one for $i\in\{1,2\}$.
\end{lem}

In \cite{Deng2014}, Deng studied the spectrality of one dimensional self-similar measure $\mu_{\rho,D_N}$ with a prime $N$ and gave a sufficient and necessary condition for the existence of an infinite set of orthogonal exponential functions in $L^{2}(\mu_{\rho,D_N})$. This was later extended to arbitrary integer $N\geq2$ by Wang {\it et al.}\cite{Wang-Wang-Dong-Zhang2018}.
\begin{thm}[\cite{Wang-Wang-Dong-Zhang2018}]\label{th2}
Let $\mu_{\rho,D_N}$ be defined by (\ref{1.6}), where $0 < |\rho| < 1$ and $N \geq 2$ is an integer. Then $L^{2}(\mu_{\rho,D_N})$ contains an infinite orthonormal set of exponential functions if and only if $\rho = \pm (q/p)^{1/r}$ for some $p, q, r \in \mathbb{N}$ with $\gcd(p, q) = 1$ and $\gcd(p, N) > 1$.
\end{thm}

For any one dimensional self-similar measure $\mu_{\rho,D}$ generated by an IFS with equal weights, associated with a general finite set $D$, a necessary condition for $\mu_{\rho,D}$ to be a spectral measure under the assumption that $\mathcal{Z}(m_{D})\subset \alpha\mathbb{Z}$ for $\alpha\in \mathbb{R}\setminus\{0\}$ is given in \cite{An-Wang2021}.
\begin{thm}[\cite{An-Wang2021}]\label{th3}
Let $0<\rho<1$ and let $D\subset\mathbb{R}$ be a finite set. Assume that $\mathcal{Z}(m_{D})$ is contained in a lattice set. If $\mu_{\rho,D}$ is a spectral measure, then $\rho^{-1}\in\mathbb{N}$.
\end{thm}

Let $\mu_1$ and $\mu_2$ be both Borel probability measures with compact support on $\mathbb{R}$ and not Dirac measures. For the measures $\mu_1$, $\mu_2$ and $\mu:=\mu_1*\mu_2$, An and Wang \cite{An-Wang2021} first proposed an assumption ($\star$) with respect to ($\mu_1,\mu_2$) as follows:
$$(\star) \ \ \rm{if \ \lambda,\gamma\in\mathcal{Z}(\hat{\mu}_2)\setminus\mathcal{Z}(\hat{\mu}_1) \ and \  \lambda-\gamma\in\mathcal{Z}(\hat{\mu}), \  then \ \lambda-\gamma\in\mathcal{Z}(\hat{\mu}_2)\setminus\mathcal{Z}(\hat{\mu}_1)}.$$
The following theorem reflects a relationship on spectrality among $\mu_1,\mu_2$ and $\mu$. Suppose that $\mu$ is a spectral measure. If the assumption ($\star$) holds with respect to ($\mu_1,\mu_2$), then the spectrum of $\mu$ can be classified by using the maximal orthogonal set of $\mu_1$. Specifically, let $0\in\Lambda$ be a spectrum of $\mu$, and $\mathcal{A}\subset\Lambda$ a maximal orthogonal set for $\mu_1$ with $0\in\mathcal{A}$. For each $\alpha\in\mathcal{A}$, define
\begin{align}\label{3.11}
\Lambda_{\alpha}:=\{\lambda\in\Lambda:\lambda-\alpha\in
\mathcal{Z}(\hat{\mu}_2)\setminus\mathcal{Z}(\hat{\mu}_1)\}\cup\{\alpha\}.
\end{align}
Then from \cite[Theorem 3.2]{Wu2024}, $\Lambda$ can be decomposed into the following disjoint union:
\begin{align}\label{3.8}
\Lambda=\bigcup_{\alpha\in\mathcal{A}}\Lambda_{\alpha}.
\end{align}

\begin{thm}[\cite{Wu2024}]\label{th4}
Let $\mu=\mu_1\ast\mu_2$, and let the assumption ($\star$) with respect to ($\mu_1,\mu_2$) holds. If $\mu$ is a spectral measure, then $\mu_2$ is a spectral measure. In particular, if $0\in\Lambda$ is a spectrum of $\mu$, then each $\Lambda_{\alpha}$ defined by \eqref{3.11} is a spectrum of  $\mu_2$.
\end{thm}

Strichartz \cite{Strichartz2000} introduced the concept of a Hadamard triple, which has become an essential tool in the study of spectral measures.

\begin{definition}[Hadamard triple,\cite{Strichartz2000}]
Let $M\in M_n(\mathbb{Z})$ be an expanding  matrix and $D, L\subset\mathbb{Z}^n$ be two finite digit sets with the same cardinality (i.e., $\#D=\#L$). We say that ($M,D,L$) is a Hadamard triple (or ($M, D$) is an admissible pair) if the matrix
\begin{equation*}
H=\frac{1}{\sqrt{\#D}}\left(e^{2\pi i\langle M^{-1}d,\ell\rangle}\right)_{d\in D,\ell\in L}
\end{equation*}
is unitary, i.e., $H^{*}H=I$, where $H^{*}$ means the transposed conjugate of $H$ and $I$ is the identity matrix.
\end{definition}

The following conclusion is a universal test for verifing the existence of a Hadamard triple.

\begin{lem}[\cite{Dutkay-Haussermann-Lai_2019}]\label{lem2-6}
Let $M\in M_n(\mathbb{Z})$ be an expanding matrix, and let $D,L\subset\mathbb{Z}^n$ be two finite digit sets with the same cardinality. Then $(M,D,L)$ is a Hadamard triple if and only if $m_D(M^{*-1}(l_1-l_2))=0$ for any $l_1\neq l_2\in L$.
\end{lem}

It is well-known that Hadamard triple is a sufficient condition for a self-affine measure to be a spectral measure\cite{Dutkay-Haussermann-Lai_2019}. However, it is not a necessary condition. A canonical example is the self-similar measure $\mu_{\rho,D}$ generated by the contraction ratio $\rho=\frac{1}{4}$ and the digit set $D=\{0,1,8,9\}$. Recently, An {\it et al.} \cite{An-Lai2023} proposed the definition of a product-form Hadamard triple, and they showed that the product-form Hadamard triple is also a sufficient condition to be a spectral measure on $\mathbb{R}$.  Although $(4,D)$ is not an admissible pair, a product-form Hadamard triple exists for it .
\begin{definition}[Higher stage product-form, \cite{An-Lai2023}]\label{Higher-stage-product-form}
We say that $(N,D,L_0\oplus\cdots\oplus L_k)$ is a (k-stage) product-form Hadamard triple if there exist sequences of positive integers $\{l_i\}_{i=1}^{k}$ such that $D=D^{(k)}$ is generated in the following process:
$$\left\{\begin{aligned}
D^{(0)}&=\mathcal{E}_0\\
D^{(1)}&=\cup_{d_0\in D^{(0)}}(d_0+N^{l_1}\mathcal{E}_1(d_0))\\
&\vdots\\
D^{(k)}&=\cup_{d_{k-1}\in D^{(k-1)}}(d_{k-1}+N^{l_1+\cdots+l_k}\mathcal{E}_k(d_{k-1}))
\end{aligned}\right.$$
and the following conditions hold for $\mathcal{E}_j(d_{j-1})$ and $L_j$:

(i) $(N,\mathcal{E}_0,L_0)$ and $(N,\mathcal{E}_j(d),L_j)$ are Hadamard triples for all $d\in D^{(j-1)}$, $j=1,2,\cdots,k$.

(ii) For all $1\leq m\leq k$,
$$(N,\mathcal{E}_0\oplus \mathcal{E}_1(d_0)\oplus\cdots\oplus\mathcal{E}_m(d_{m-1}),L_0\oplus L_1\oplus\cdots\oplus L_m)$$
and
$$(N,\mathcal{E}_m(d_{m-1})\oplus \mathcal{E}_{m+1}(d_{m})\oplus\cdots\oplus\mathcal{E}_k(d_{k-1}),L_m\oplus L_{m+1}\oplus\cdots\oplus L_k)$$
are Hadamard triples for all $d_j\in D^{(j)}$, $j=1,2,\cdots,k-1$.
\end{definition}

\begin{thm}[\cite{An-Lai2023}]\label{th5}
Let $(N,D,L_0\oplus\cdots\oplus L_k)$ be a k-stage product-form Hadamard triple. Then the self-similar measure $\mu_{N,D}$ is a spectral measure.
\end{thm}

Recently, Luo {\it et al.} \cite{Luo-Mao-Liu2025} introduced a new class of Moran measures   $\mu_{\{b_k\},\{R_k\}}$ generated by a sequence of rational number $\{b_k:=\frac{l_k}{t_k}\}_{k=1}^{\infty}$ and a sequence of consecutive digit sets $\{R_k\}_{k=1}^{\infty}$ with $R_k=\{0,1,\cdots,\gamma_k-1\}$.
Under the condition that both $\{t_k\}_{k=1}^{\infty}$ and $\{\gamma_k\}_{k=1}^{\infty}$ are bounded, combined with \cite[(5.4)]{Luo-Mao-Liu2025}, some results are organised as follows. Let $c\in\mathbb{N}$ be the common multiple of $\{t_k\}_{k=1}^{\infty}$ and $\{\gamma_k\}_{k=1}^{\infty}$ and $q_k:=\frac{c}{\gamma_k}$. It is easy to see that $\mathcal{Z}(\hat{\mu}_{\{b_k\},\{R_k\}})\subset\frac{b_1}{c}\mathbb{Z}$. Suppose that $0\in\Lambda$ is a spectrum of $\mu_{\{b_k\},\{R_k\}}$. Then we have $\frac{1}{b_1}\Lambda\subset\frac{\mathbb{Z}}{c}$. Note that for any $z\in\mathbb{Z}$, there exist unique $i\in\{0,1,\cdots,q_1-1\}$, $j\in\{0,1,\cdots,\gamma_1-1\}$ and $z'\in\mathbb{Z}$ such that
$z=i+q_1j+cz'$. Thus one may obtain the following decomposition
\begin{align}\label{2.5}
\frac{1}{b_1}\Lambda=\frac{1}{c}\bigcup_{i=0}^{q_1-1}\bigcup_{j=0}^{\gamma_1-1}(i+q_1j+c\Lambda_{i+q_1j})=\bigcup_{i=0}^{q_1-1}\bigcup_{j=0}^{\gamma_1-1}(\frac{i+q_1j}{c}+\Lambda_{i+q_1j}),
\end{align}
where $\Lambda_{i+q_1j}=\mathbb{Z}\cap\left(\frac{\Lambda}{b_1}-\frac{i+q_1j}{c}\right)$ and $\frac{i+q_1j}{c}+\Lambda_{i+q_1j}=\emptyset$ if $\Lambda_{i+q_1j}=\emptyset$.
\begin{lem}[ \cite{Luo-Mao-Liu2025}]\label{pro7}
Let $\mu_{\{b_k\},\{R_k\}}$ defined by (\ref{1.4}) be a spectral measure with $b_k=\frac{l_k}{t_k}$ and $R_k=\{0,1,\cdots,\gamma_k-1\}$. If $\{t_k\}_{k=1}^{\infty}$ and $\{\gamma_k\}_{k=1}^{\infty}$ are bounded and $0\in\Lambda$ is a spectrum of $\mu_{\{b_k\},\{R_k\}}$, then
\begin{enumerate}[\rm(i)]
\item $\mu_{\{b_k,R_k,>k\}}$ defined by (\ref{1.15}) is a spectral measure for all $k\geq1$. Furthermore, if $0\in\Gamma_k$ is a spectrum of $\mu_{\{b_k,R_k,>k\}}$, then for any $\{j_i:0\leq i\leq q_{k+1}-1\}\subset\{0,1,\cdots, \gamma_{k+1}-1\}$, the set
$$\Gamma_{k+1}=\bigcup_{i=0}^{q_{k+1}-1}\left(\frac{i+q_{k+1}j_i}{c}+\Lambda_{i+q_{k+1}j_i}\right)$$
is a spectrum of $\mu_{\{b_k,R_k,>k+1\}}$ if $\Gamma_{k+1}\neq\emptyset$, where $\Lambda_{i+q_{k+1}j_i}=\mathbb{Z}\cap\left(\frac{\Gamma_k}{b_{k+1}}-\frac{i+q_{k+1}j_i}{c}\right)$.

\item For any $i\in\{0,1,\cdots, q_k-1\}$ with $k\geq 1$,  either $\Lambda_{i+q_kj}\neq\emptyset$ for all $j\in\{0,1,\cdots,\gamma_k-1\}$ or
$\Lambda_{i+q_kj}=\emptyset$ for all $j\in\{0,1,\cdots,\gamma_k-1\}$.

\item If $\gamma_k \nmid t_{k+1}$ for some $k\geq 1$, then $\gamma_{k+1} \mid l_{k+1}$.
\end{enumerate}
\end{lem}

It is shown in \cite{Dutkay-Jorgensen_2019} that the spectral properties of measures is invariant under a similarity transformation, which can be exploited to simplify the object of our study.

\begin{lem}[\cite{Dutkay-Jorgensen_2019}]\label{lem4}
Let $M_{1},M_{2}\in M_{n}(\mathbb{R})$ be two expanding matrices, and let $D_{1},D_{2}\subset\mathbb{R}^{n}$ be two finite digit sets with the same cardinality.
If there exists a matrix $Q\in M_{n}(\mathbb{R})$ such that $M_{2}=QM_{1}Q^{-1}$ and $D_{2}=QD_{1}$, then
$\mu_{M_{1}^{-1},D_{1}}$ is a spectral measure with a spectrum $\Lambda$ if and only if
$\mu_{M_{2}^{-1},D_{2}}$ is a spectral measure with a spectrum $Q^{*-1}\Lambda$.
\end{lem}

At the end of this section, we show that the measure $\nu_{\rho,D_{2Nm}}^{m}$ is actually equal to the canonical self-similar measure $\mu_{\rho,D}$ generated by (\ref{1.6}), associated with the digit set $D$ denoted by (\ref{1.11}).

\begin{pro}\label{lem3}
Let $\nu_{\rho,D_{2Nm}}^{m}$ and $\mu_{\rho,D}$ be the self-similar measures generated by the IFS in (\ref{1.3}) and (\ref{1.6}) respectively, where $0<\rho<1$, $D_{2Nm}=\{0,1,\cdots,2Nm-1\}$ and $D$ is defined by (\ref{1.11}).
Then $\nu_{\rho,D_{2Nm}}^{m}=\mu_{\rho,D}$.
\end{pro}

\begin{proof}
Note that the corresponding IFS of $\nu:=\nu_{\rho,D_{2Nm}}^{m}$ is $\{\tau_d(\cdot)\}_{d\in D_{2Nm}}$, where
$$\tau_d(\cdot)=(-1)^{\lfloor\frac{d}{m}\rfloor}\rho(\cdot+d).$$
Substituting the above equation into (\ref{1.2}), the Fourier transform of $\nu$ can be derived as follows:
\begin{align}\label{3.1}
\hat{\nu}(t)&=\frac{1}{2Nm}\left(\sum_{d=0}^{m-1}e^{2\pi id\rho t}\hat{\nu}(\rho t)+\sum_{d=m}^{2m-1}e^{-2\pi id\rho t}\hat{\nu}(-\rho t)+\cdots+\sum_{d=(2N-1)m}^{2Nm-1}e^{-2\pi id\rho t}\hat{\nu}(-\rho t)\right)\nonumber\\
&=\frac{1}{2Nm}\left(\left(\sum_{d=0}^{m-1}+\cdots+\sum_{d=(2N-2)m}^{(2N-1)m-1}\right)e^{2\pi id\rho t}\hat{\nu}(\rho t)+\left(\sum_{d=m}^{2m-1}+\cdots+\sum_{d=(2N-1)m}^{2Nm-1}\right)e^{-2\pi id\rho t}\hat{\nu}(-\rho t)\right)\nonumber\\
&=\frac{1}{2Nm}\left(\sum_{d=0}^{m-1}+\sum_{d=2m}^{3m-1}+\cdots+\sum_{d=(2N-2)m}^{(2N-1)m-1}\right)\left(e^{2\pi id\rho t}\hat{\nu}(\rho t)+e^{-2\pi i(d+m)\rho t}\hat{\nu}(-\rho t)\right).
\end{align}
By using a substitution, taking $-t=t$ yields
\begin{align*}
\hat{\nu}(-t)&=\frac{1}{2Nm}\left(\sum_{d=0}^{m-1}+\sum_{d=2m}^{3m-1}+\cdots+\sum_{d=(2N-2)m}^{(2N-1)m-1}\right)\left(e^{-2\pi id\rho t}\hat{\nu}(-\rho t)+e^{2\pi i(d+m)\rho t}\hat{\nu}(\rho t)\right)\\
&=e^{2\pi im\rho t}\hat{\nu}(t),
\end{align*}
which establishes the relationship between $\hat{\nu}(t)$ and $\hat{\nu}(-t)$. Then, combining with (\ref{3.1}), we can obtain the following result
\begin{align*}
\hat{\nu}(t)&=\frac{1}{2Nm}\left(\left(\sum_{d=0}^{m-1}+\cdots+\sum_{d=(2N-2)m}^{(2N-1)m-1}\right)e^{2\pi id\rho t}\hat{\nu}(\rho t)+\left(\sum_{d=m}^{2m-1}+\cdots+\sum_{d=(2N-1)m}^{2Nm-1}\right)e^{-2\pi id\rho t}\hat{\nu}(-\rho t)\right)\nonumber\\
&=\frac{1}{2Nm}\left(1+e^{-2\pi i(2Nm-1-m\rho)\rho t}\right)\left(\sum_{d=0}^{m-1}+\sum_{d=2m}^{3m-1}+\cdots+\sum_{d=(2N-2)m}^{(2N-1)m-1}\right)e^{2\pi id\rho t}\hat{\nu}(\rho t)\nonumber\\
&=\frac{1}{2Nm}\left(1+e^{2\pi i(1+m\rho-2Nm)\rho t}\right)\left(\sum_{d=0}^{m-1}e^{2\pi id\rho t}\right)\left(\sum_{d=0}^{N-1}e^{2\pi i2md\rho t}\right)\hat{\nu}(\rho t).
\end{align*}
On the other hand, by using (\ref{1.13}) and (\ref{1.6}), a direct calculation gives that
\begin{align*}
\hat{\mu}_{\rho,D}(t)
=\frac{1}{2Nm}\left(1+e^{2\pi i(1+m\rho-2Nm)\rho t}\right)\left(\sum_{d=0}^{m-1}e^{2\pi id\rho t}\right)\left(\sum_{d=0}^{N-1}e^{2\pi i2md\rho t}\right)\hat{\mu}_{\rho,D}(\rho t).
\end{align*}
Then it is clear that $\hat{\mu}_{\rho,D}(t)=\hat{\nu}(t)$. 
Hence $\mu_{\rho,D}=\nu$ by the uniqueness of Fourier transform.
\end{proof}

\begin{re}\label{rem2}
Combining Proposition \ref{lem3} and Lemma \ref{pro6}, one may obtain that a set $\Lambda\subset\mathbb{R}$ is a spectrum of $\nu_{\rho,D_{2Nm}}^{m}$ if and only if $\Lambda$ is a spectrum of $\mu_{\rho,D}$. Therefore, the spectrality of $\nu_{\rho,D_{2Nm}}^{m}$ (Theorem \ref{thm1.1}) can be established by characterizing the spectrality of $\mu_{\rho,D}$ (Theorem \ref{thm1.2}).
\end{re}

\section{Spectrality of self-similar measures\label{sec3}}
Our purpose in this section is to study the spectrality of $\mu_{\rho,D}$, and prove Theorem \ref{thm1.2}. For convenience, we write
$$D=D_m\oplus 2mD_N \oplus (1+m\rho-2Nm)D_2
:=\mathcal{D}_1\oplus \mathcal{D}_2\oplus \mathcal{D}_3.$$
It is worth noting that we always assume $m,N \geq 2$. When $m = 1$, Theorem \ref{thm1.2} follows directly from Theorem \ref{thm1.0}; when $N = 1$, the proof is analogous to the case $N \geq 2$, so we omit the details for brevity.

For any integer $r\geq1$, define
$$\mathcal{Q}^{\frac{1}{r}}:=\{\rho=u^{\frac{1}{r}}: 0<u<1,u\in\mathbb{Q}\},$$
where $r$ is the smallest integer for $\rho$ and $u$ is the simplest form.
In the proof of Theorem \ref{thm1.2}, the direction of `$(i)\Rightarrow(ii)$ ' is the most tedious.
Abbreviating ``$\mu_{\rho,D}$ is a spectral measure'' by ``spectral'', we establish this direction via the following four steps:
\begin{enumerate}[(A)]
\item ``spectral $\Rightarrow \rho\in\mathcal{Q}^{\frac{1}{r}}$, i.e., $\rho=(\frac{q}{p})^{\frac{1}{r}}$ for some $ p,q,r\in\mathbb{N}^+$ with $\gcd{(p,q)}=1$ and $1\leq q<p$'' (Subsection 3.1);

\item ``spectral $\Rightarrow q=1$, i.e., $\rho=(\frac{1}{p})^{\frac{1}{r}}$ for some $p,r\in\mathbb{N}^+$ with $p>1$'' (Subsection 3.1);

\item ``spectral $\Rightarrow r=1$, i.e., $\rho=\frac{1}{p}$ for some $p\in\mathbb{N}^+$ with $p>1$'' (Subsection 3.2);

\item ``spectral $\Rightarrow 2Nm\mid p$, i.e., $\rho=\frac{1}{2Nmp'}$ for some $p'\in\mathbb{N}^+$'' (Subsection 3.3).
\end{enumerate}
In subsection 3.4, a complete proof of Theorem \ref{thm1.2}  will be given.

\subsection{The steps (A) and (B)\label{sec3.1}}
The realisations of steps (A) and (B) correspond to Propositions \ref{pro1} and \ref{pro2}, respectively.
\begin{pro}\label{pro1}
If $\mu_{\rho,D}$ is a spectral measure with a real number $\rho\in(0,1)$, then $\rho\in\mathcal{Q}^{\frac{1}{r}}$.
\end{pro}
\begin{proof}
Assume that $\mu_{\rho,D}$ is a spectral measure. Then $L^{2}(\mu_{\rho,D})$ contains an infinite orthonormal set of exponential functions. Note that $\mu_{\rho,D}$ can be expressed as
$$\mu_{\rho,D}=\mu_{\rho,\mathcal{D}_1}\ast\mu_{\rho,\mathcal{D}_2}\ast\mu_{\rho,\mathcal{D}_3}.$$
It follows from Lemma \ref{lem1} that one of $L^2(\mu_{\rho,\mathcal{D}_1})$, $L^2(\mu_{\rho,\mathcal{D}_2})$ and $L^2(\mu_{\rho,\mathcal{D}_3})$ has an infinite orthonormal set. Without loss of generality, assume that $L^2(\mu_{\rho,\mathcal{D}_3})$ has an infinite orthonormal set. By a similarity transformation, we conclude that $L^2(\mu_{\rho,D_2})$ has an infinite orthonormal set.
Then from Theorem \ref{th2}, it is known that
$\rho=(\frac{q}{p})^{\frac{1}{r}}$ for some $p,q,r\in\mathbb{N}^{+}$ with $\gcd{(p,q)}=1$, i.e., $\rho\in\mathcal{Q}^{\frac{1}{r}}$.
Hence the proof is completed.
\end{proof}
Next, we continue to consider the case  $\rho\in\mathcal{Q}^{\frac{1}{r}}$, i.e., $\rho=(\frac{q}{p})^{\frac{1}{r}}$ for some $ p,q,r\in\mathbb{N}^+$ with $\gcd{(p,q)}=1$ and $1\leq q<p$. If  $r\geq2$, the measure $\mu_{\rho,D}$ can be expressed as follows:
\begin{align*}
\mu_{\rho,D}&=\mu_{\rho,\mathcal{D}_1}\ast\mu_{\rho,\mathcal{D}_2}\ast\mu_{\rho,\mathcal{D}_3}\\
&=(\ast_{j=0}^{\infty}\ast_{s=1,s\neq l}^{r}\delta_{\rho^{rj+s}
\mathcal{D}_1})\ast(\ast_{j=0}^{\infty}\ast_{s=1, s\neq l}^{r}
\delta_{\rho^{rj+s}\mathcal{D}_2})\ast\mu_{\rho,\mathcal{D}_3}*(\ast_{j=0}^{\infty}\delta_{\rho^{rj+l}\mathcal{D}_1})\ast(\ast_{j=0}^{\infty}\delta_{\rho^{rj+l}
\mathcal{D}_2})\\
&:=\mu_{1,l}\ast\mu_{2,l}
\end{align*}
for $1\leq l\leq r$,
where
\begin{equation}\label{2-1}
\mu_{1,l}=(\ast_{j=0}^{\infty}\ast_{s=1,s\neq l}^{r}\delta_{\rho^{rj+s}\mathcal{D}_1})\ast(\ast_{j=0}^{\infty}\ast_{s=1,s\neq l}^{r}\delta_{\rho^{rj+s}\mathcal{D}_2})\ast\mu_{\rho,\mathcal{D}_3} \end{equation}
and
\begin{equation}\label{2-2}
\mu_{2,l}=(\ast_{j=0}^{\infty}\delta_{\rho^{rj+l}\mathcal{D}_1})
\ast(\ast_{j=0}^{\infty}\delta_{\rho^{rj+l}\mathcal{D}_2}).
\end{equation}
For simplicity, we let $u=\rho^{r}\in\mathbb{Q}$. A direct calculation gives
\begin{align}\label{3.3.0}
\mathcal{Z}(\hat{\mu}_{1,l})&=\bigcup_{j=0}^{\infty}\bigcup_{s=1, s\neq l}^{r}\frac{1}{u^{j}\rho^{s}}\left(\frac{\mathbb{Z}\setminus m\mathbb{Z}}{m}\cup\frac{\mathbb{Z}\setminus N\mathbb{Z}}{2mN}\right)\cup\bigcup_{j=1}^{\infty}\frac{2\mathbb{Z}+1}{2\rho^{j}(1+m\rho-2Nm)}\nonumber\\
&\subset\left(\bigcup_{j=0}^{\infty}\bigcup_{s=1,s\neq l}^{r}\frac{\mathbb{Z}
\setminus\{0\}}{u^{j}\rho^{s}2mN}\right)
\cup\left(\bigcup_{j=1}^{\infty}\frac{2\mathbb{Z}+1}{2\rho^{j}(1+m\rho-2Nm)}\right)
:=\mathcal{Z}_{1,l}^1\cup\mathcal{Z}_{1,l}^2
\end{align}
and
\begin{align}\label{3.4.0}
\mathcal{Z}(\hat{\mu}_{2,l})=
\bigcup_{j=0}^{\infty}\frac{1}{u^{j}\rho^{l}}\left(\frac{\mathbb{Z}\setminus m\mathbb{Z}}{m}\cup\frac{\mathbb{Z}\setminus N\mathbb{Z}}{2mN}\right)\subset
\bigcup_{j=0}^{\infty}\frac{\mathbb{Z}\setminus\{0\}}{u^{j}\rho^{l}2mN}:=\mathcal{Z}_{2,l}.
\end{align}

The relationship between the two measures $\mu_{1,l}$, $\mu_{2,l}$ is given in the following lemma.
\begin{lem}\label{lem2-1}
For any $1\leq l\leq r$, let $\mu_{1,l}$ and $\mu_{2,l}$ be defined by (\ref{2-1}) and (\ref{2-2}), respectively. If $r\geq 2$, then the assumption ($\star$) holds with respect to $(\mu_{1,l},\mu_{2,l})$.
\end{lem}
\begin{proof} We first show that $\mathcal{Z}(\hat{\mu}_{1,l})\cap\mathcal{Z}(\hat{\mu}_{2,l})=\emptyset$ by contradiction. Suppose that there exist $\lambda_1\in \mathcal{Z}(\hat{\mu}_{1,l})$ and $\lambda_2\in \mathcal{Z}(\hat{\mu}_{2,l})$ such that $\lambda_1=\lambda_2$, then by \eqref{3.3.0} and  \eqref{3.4.0}, there exist nonzero integers $a_1, a_2, a_3$ and nonnegative integers $j_1, j_2, j_3$ ($j_3\geq 1$) such that
$$
\frac{a_2}{u^{j_2}\rho^{s_0}2mN}=\frac{a_1}{u^{j_1}\rho^{l}2mN}
$$
or
$$
\frac{a_3}{2\rho^{j_3}(1+m\rho-2Nm)}=\frac{a_1}{u^{j_1}\rho^{l}2mN},
$$
where $1\leq s_0\leq r$ and  $s_0\neq l$.

Let $j_3=rk+s'$ with a nonnegative integer $k$ and $1\leq s'\leq r$, then we have
$$
a_1u^{j_2}\rho^{s_0}-a_2u^{j_1}\rho^{l} =0
$$
or
$$
a_1m u^k\rho^{s'+1}+a_1(1-2Nm)u^k\rho^{s'}-a_3\rho^{l}u^{j_1}mN=0.
$$
Notice that $1\leq s_0, l, s'\leq r$ with $s_0\neq l$ and none of the coefficients of $\rho^{s'+1}$, $\rho^{s'}$, or $\rho^{l}$ vanish. It can be easily verified that the above two equations contradict with the fact that $x^{r}-u\in\mathbb{Q}[x]$ is the minimal polynomial of $\rho$. Therefore,  $\mathcal{Z}(\hat{\mu}_{1,l})\cap\mathcal{Z}(\hat{\mu}_{2,l})=\emptyset$.

Having obtained $\mathcal{Z}(\hat{\mu}_{1,l})\cap\mathcal{Z}(\hat{\mu}_{2,l})=\emptyset$, to show the assumption ($\star$) holds with respect to $(\mu_{1,l},\mu_{2,l})$,  we only need to show that
$$
(\mathcal{Z}(\hat{\mu}_{2,l})-\mathcal{Z}(\hat{\mu}_{2,l}))\cap\mathcal{Z}(\hat{\mu}_{1,l})=\emptyset.
$$
Since $0\notin \mathcal{Z}_{1,l}^1\cup\mathcal{Z}_{1,l}^2$,
it is sufficient to show that $((\mathcal{Z}_{2,l}-\mathcal{Z}_{2,l})\setminus\{0\})\cap(\mathcal{Z}_{1,l}^1\cup\mathcal{Z}_{1,l}^2)=\emptyset$.
Otherwise,  there exist $\lambda_1\neq\lambda_2\in\mathcal{Z}_{2,l}$ such that $\lambda_1-\lambda_2=\lambda\in\mathcal{Z}_{1,l}^1\cup\mathcal{Z}_{1,l}^2.$ Hence,
there exist $a'_1,a'_2,a'_3, a'_4\in\mathbb{Z}\setminus\{0\}$, $j'_1,j'_2, j'_3\geq0$ and $j'_4\geq 1$  such that
$$
\frac{a'_1}{u^{j'_1}\rho^{l}2mN}-\frac{a'_2}{u^{j'_2}\rho^{l}2mN}=\frac{a'_3}{u^{j'_3}\rho^{s'_0}2mN}
$$
or
$$
\frac{a'_1}{u^{j'_1}\rho^{l}2mN}-\frac{a'_2}{u^{j'_2}\rho^{l}2mN}=\frac{a'_4}{2\rho^{j'_4}(1+m\rho-2Nm)}.
$$
where $1\leq s'_0\leq r$ and $s'_0\neq l$. Similarly, a  contradiction emerges from the minimal polynomial of $\rho$. Therefore, the assumption ($\star$) holds with respect to $(\mu_{1,l},\mu_{2,l})$.
\end{proof}

Let $\bar{D}=\mathcal{D}_1\oplus \mathcal{D}_2$. The self-similar measure $\mu_{\rho,D}$ can also be expressed as
\begin{align*}
\mu_{\rho,D}=\mu_{\rho,\bar{D}}\ast\mu_{\rho,\mathcal{D}_3},
\end{align*}
then it is known easily that
\begin{align}\label{2-4}
\mathcal{Z}(\hat{\mu}_{\rho,\bar{D}})=\bigcup_{j=1}^{\infty}\frac{1}{\rho^{j}}\left(\frac{\mathbb{Z}\setminus m\mathbb{Z}}{m}\cup\frac{\mathbb{Z}\setminus N\mathbb{Z}}{2mN}\right) \ \ {\rm{and}} \ \
\mathcal{Z}(\hat{\mu}_{\rho,\mathcal{D}_3})=\bigcup_{j=1}^{\infty}\frac{2\mathbb{Z}+1}{2\rho^{j}(1+m\rho-2Nm)}.
\end{align}
The following result shows that ``$q=1$'' is a necessary condition for $\mu_{\rho,D}$ to be a spectral measure, where $\rho=(\frac{q}{p})^{\frac{1}{r}}$ for some $ p,q,r\in\mathbb{N}^+$ with $\gcd{(p,q)}=1$ and $1\leq q<p$.
\begin{pro}\label{pro2}
If $\mu_{\rho,D}$ is a spectral measure, where $\rho=(\frac{q}{p})^{\frac{1}{r}}$ for some $ p,q,r\in\mathbb{N}^+$ with $\gcd{(p,q)}=1$ and $1\leq q<p$, then $q=1$.
\end{pro}
\begin{proof}
For the case $r\geq 2$, let $\mu_{\rho,D}=\mu_{1,r}\ast\mu_{2,r}$ be a spectral measure, where $\mu_{1,r}$ and $\mu_{2,r}$ are defined by (\ref{2-1}) and (\ref{2-2}) respectively. Combining Lemma \ref{lem2-1} and  Theorem \ref{th4}, we obtain that $\mu_{2,r}$ is a spectral measure. From (\ref{2-1}), it is easy to see that
\begin{align*}
\mu_{2,r}=(\ast_{j=1}^{\infty}\delta_{u^{j}\mathcal{D}_1})\ast(\ast_{j=1}^{\infty}\delta_{u^{j}\mathcal{D}_2})
=\delta_{u\bar{D}}*\delta_{u^{2}\bar{D}}*\delta_{u^{3}\bar{D}}*\cdots=\mu_{u,\bar{D}},
\end{align*}
where $u=\rho^{r}=\frac{q}{p}\in (0,1)$. Note that $\mathcal{Z}(m_{\bar{D}})$ is contained in a lattice set. Then it follows from Theorem \ref{th3} that $u^{-1}=\frac{p}{q}\in\mathbb{N}$, which shows that $q=1$ by $\gcd{(p,q)}=1$. Likewise, for the case $r=1$, we observe that the set
$$\mathcal{Z}(m_{D})=\frac{\mathbb{Z}\setminus m\mathbb{Z}}{m}\cup\frac{\mathbb{Z}\setminus N\mathbb{Z}}{2mN}\cup\frac{2\mathbb{Z}+1}{2(1+m\rho-2Nm)}$$
is contained in a lattice set. Combining this fact with the spectrality of $\mu_{\rho,D}$, we get $q=1$ by using Theorem \ref{th3} again.
\end{proof}
\subsection{The step (C)\label{sec3.2}}
By subsection \ref{sec3.1}, $\rho$ can be written as $\rho=(\frac{1}{p})^{\frac{1}{r}}$ with $p>1$ and $r\geq1$ if $\mu_{\rho,D}$ is a spectral measure. In this subsection, we further show that ``$r=1$'' is a necessary condition for $\mu_{\rho,D}$ to be a spectral measure. To establish this result, we argue by contradiction\textemdash assuming $r\geq2$ leads to a contradiction, as shown in Proposition \ref{pro3'}. Before presenting the main argument, we first introduce some necessary preparations, including two key lemmas.
Let $\{\mu_{2,l}\}_{l=1}^{r}$  be the self-similar measures defined by \eqref{2-1}. Since $\bar{D}=\mathcal{D}_1\oplus \mathcal{D}_2$, the measure $\mu_{2,l}$ admits the following representation:
\begin{equation*}
\mu_{2,l}=\ast_{j=0}^{\infty}\delta_{\rho^{l}u^{j}\bar{D}}.
\end{equation*}

Below we investigate the structure of the spectrum of $\mu_{2,l}$ for $l=1,2,\cdots,r$. We only analyse the case $l = r$ in detail (see Remark \ref{rem1} for the case $l\neq r$). Write $\mu_{2,r}$ in the form of a Moran measure as follows.
\begin{align*}
\mu_{2,r}&=\delta_{p^{-1}\bar{D}}\ast\delta_{p^{-2}\bar{D}}\ast\delta_{p^{-3}\bar{D}}\ast\cdots\\
&=\delta_{p^{-1}2mD_N}\ast\delta_{p^{-1}D_{m}}\ast\delta_{p^{-2}2mD_N}\ast\delta_{p^{-2}D_{m}}\ast\cdots\\
&:=\delta_{b_{1}'^{-1}R_{1}'}\ast\delta_{b_{1}'^{-1}b_{2}'^{-1}R_{2}'}\ast\delta_{b_{1}'^{-1}b_{2}'^{-1}b_{3}'^{-1}R_{3}'}\ast\cdots,
\end{align*}
where $b'_{2k-1}=\frac{p}{2m}$, $b_{2k}'=2m$ and $R_{2k-1}'=D_N$, $R_{2k}'=D_m$ for all $k\geq1$. Since both $\{b_k'\}_{k=1}^{\infty}$ and $\{R_k'\}_{k=1}^{\infty}$ only have a finite number of choices, we can quickly characterise the structure of the spectrum of $\mu_{2,r}$ by using Lemma \ref{pro7}.
Let $c=\rm{lcm}$$(2m,N)\in\mathbb{N}$, $d_{2k-1}=N$ and $d_{2k}=m$ for $k\geq1$. Write $q_{2k-1}:=\frac{c}{d_{2k-1}}$ and $q_{2k}:=\frac{c}{d_{2k}}$ for $k\geq1$. Suppose that $0\in\Lambda$ is a spectrum of $\mu_{2,r}$. Then it follows from (\ref{2.5}) that $\Lambda$ has the following decomposition
\begin{align}\label{2-5}
\Lambda=\frac{b_1'}{c}\bigcup_{i=0}^{q_1-1}\bigcup_{j=0}^{N-1}(i+q_1j+c\Lambda_{i+q_1j})
\end{align}
where $\Lambda_{i+q_1j}=\mathbb{Z}\cap\left(\frac{\Lambda}{b_1'}-\frac{i+q_1j}{c}\right)$
and $\frac{i+q_1j}{c}+\Lambda_{i+q_1j}=\emptyset$ if $\Lambda_{i+q_1j}=\emptyset$.
By using Lemma \ref{pro7} (i) and (ii), it is known that $\mu_{>k}:=\mu_{\{b_k',R_k',>k\}}$ defined by (\ref{1.15}) is a spectral measure for $k\geq1$ and  the spectrum of  $\mu_{>k}$ has the following properties.

\begin{lem}\label{lem2-5}
Suppose that $0\in\Gamma_{k-1}$ is a spectrum of $\mu_{>{k-1}}$ for $k\geq 1$, where $\mu_{>0}=\mu_{2,r}$. Then we have the following conclusions.
\begin{enumerate}[\rm(i)]
\item For any $\{j_i:0\leq i\leq q_k-1\}\subset\{0,1,\cdots, d_k-1\}$, the set
$$\Gamma_{k}=\bigcup_{i=0}^{q_k-1}\left(\frac{i+q_kj_i}{c}+\Lambda_{i+q_kj_i}\right)$$
is a spectrum of $\mu_{>k}$ if $\Gamma_{k}\neq\emptyset$, where $\Lambda_{i+q_kj_i}=\mathbb{Z}\cap\left(\frac{\Gamma_{k-1}}{b_k^{'}}-\frac{i+q_kj_i}{c}\right)$.

\item For any $i\in\{0,1,\cdots, q_k-1\}$ with $k\geq 1$,  either $\Lambda_{i+q_kj}\neq\emptyset$ for all $j\in\{0,1,\cdots,d_k-1\}$ or
$\Lambda_{i+q_kj}=\emptyset$ for all $j\in\{0,1,\cdots,d_k-1\}$.
\end{enumerate}

\end{lem}

\begin{lem}\label{lem2-9}
Suppose that $0\in\Lambda$ is a spectrum of $\mu_{2,r}$, then for any $t\in\{1,2,\cdots, m-1\}$ there exists $z_t\in\mathbb{Z}$ such that $\frac{p(t+mz_t)}{m}\in\Lambda$.
\end{lem}

\begin{proof}
By (\ref{2-5}) and Lemma \ref{lem2-5} (i), for any group $\{j_i\}_{i=1}^{q_1-1}\subset\{0,1,\cdots, N-1\}$, the set
$$\Gamma_{j_0,j_1,\cdots,j_{q_1-1}}=\bigcup_{i=0}^{q_1-1}(\frac{i}{c}+\frac{j_i}{N}+\Lambda_{i+q_1j_i})$$
is a spectrum of $\mu_{>1}$ if  $\Gamma_{j_0,j_1,\cdots,j_{q_1-1}}\neq\emptyset$. Since $0\in\Lambda$, it follows that $\Lambda_{0+q_10}\neq\emptyset$ and $0\in\Gamma_{0,j_1,\cdots,j_{q_1-1}}:=\Gamma_{1}$ is a spectrum of $\mu_{>1}$.
Similar to the decomposition in (\ref{2-5}), $\Gamma_{1}$ can also be expressed as
$$\Gamma_{1}=b_2'\bigcup_{s=0}^{q_2-1}\bigcup_{t=0}^{m-1}\left(\frac{s+q_2t}{c}+\Lambda_{s+q_2t}'\right)$$
where
$\Lambda_{s+q_2t}'=\mathbb{Z}\cap\left(\frac{\Gamma_1}{b_2'}-\frac{s+q_2t}{c}\right)$
and $\frac{s+q_2t}{c}+\Lambda_{s+q_2t}'=\emptyset$ if $\Lambda_{s+q_2t}'=\emptyset$.
Note that $\Lambda_{0+q_20}'\neq\emptyset$ since $0\in\Gamma_1$. Using Lemma \ref{lem2-5} (ii) to $\Lambda_{s+q_2t}'$, one may obtain that $\Lambda_{0+q_2t}'\neq\emptyset$ for all $t\in\{0,1,\cdots,m-1\}$. Consequently, for any $t\in\{1,2,\cdots, m-1\}$ there exists $z_t\in\Lambda'_{0+q_2t}$ such that $\frac{b_2'(t+mz_t)}{m}\in\Gamma_{1}$. Since $b_1'\Gamma_{1}\subset\Lambda$, it follows that
$$\frac{b_1'b_2'(t+mz_t)}{m}=\frac{p(t+mz_t)}{m}\in \Lambda$$
for any $t\in\{1,2,\cdots, m-1\}$. Hence we finish the proof.
\end{proof}

\begin{re}\label{rem1}
For $\mu_{2,l}$ with $l\neq r$, it can be easily verified that $\mu_{2,r}(\cdot)=\mu_{2,l}(\rho^{l}p\cdot)$. Suppose that $0\in\Lambda_{2,l}$ is a spectrum of $\mu_{2,l}$ for each $1\leq l\leq r-1$. Combining Lemma \ref{lem4} and (\ref{2-5}), one may obtain that $\Lambda_{2,l}$ can be decomposed as follows:
$$\Lambda_{2,l}=\frac{1}{2m\rho^{l}}\bigcup_{s=0}^{q_1-1}\bigcup_{t=0}^{N-1}(\frac{s+q_1t}{c}+\Lambda_{s+q_1t})$$
where
$\Lambda_{s+q_1t}=\mathbb{Z}\cap\left(2m\rho^{l}\Lambda_{2,l}-\frac{s+q_1t}{c}\right)$
and $\frac{s+q_1t}{c}+\Lambda_{s+q_1t}=\emptyset$ if $\Lambda_{s+q_1t}=\emptyset$.
Similar to the analysis of Lemmas \ref{lem2-5} and \ref{lem2-9}, we can also conclude that for any $t\in\{1,2,\cdots, m-1\}$ there exists $z_t\in\mathbb{Z}$ such that $\frac{t+mz_t}{m\rho^{l}}\in\Lambda_{2,l}$.
\end{re}

The step (C) can be established by the following result.
\begin{pro}\label{pro3'}
If $\rho=(\frac{1}{p})^{\frac{1}{r}}$ for some $p,r\in\mathbb{N}^+$ with $p>1$ and $r\geq2$, then $\mu_{\rho,D}$ is not a spectral measure.
\end{pro}

\begin{proof}
Suppose on the contrary that $\mu_{\rho,D}$ is a spectral measure. Let $\mu_{\rho,D}=\mu_{1,l}\ast\mu_{2,l}$ for $1\leq l\leq r$, then it follows from Lemma
\ref{lem2-1} that $(\mu_{1,l},\mu_{2,l})$ satisfies the assumption ($\star$).
Let $0\in\Lambda$ be a spectrum of $\mu_{\rho,D}$, and $\mathcal{A}\subset\Lambda$  a maximal orthogonal set for $\mu_{1,l}$ with $0\in\mathcal{A}$. Then it follows from Theorem \ref{th4}  that for each $\alpha\in\mathcal{A}$,
$$
\Lambda_{\alpha}^l=\{\lambda\in\Lambda:\lambda-\alpha\in\mathcal{Z}
(\hat{\mu}_{2,l})\setminus\mathcal{Z}(\hat{\mu}_{1,l})\}\cup\{\alpha\}
$$
is a spectrum of $\mu_{2,l}$. Since $\mathcal{Z}(\hat{\mu}_{1,l})\cap\mathcal{Z}(\hat{\mu}_{2,l})=\emptyset$ by the proof of Lemma
\ref{lem2-1}, taking $\alpha=0$, we know that
$\Lambda_{0}^l:=(\Lambda\cap\mathcal{Z}(\hat{\mu}_{2,l}))\cup\{0\}$ is a spectrum of $\mu_{2,l}$. Below we arrive at a contradiction by analyzing $\Lambda_{0}^1$ and $\Lambda_{0}^2$.

Since $\Lambda_{0}^1$ and $\Lambda_{0}^2$ are spectra of $\mu_{2,1}$ and $\mu_{2,2}$  respectively,
Lemma \ref{lem2-9} and Remark \ref{rem1} imply that there exist $z_1, z_2\in \mathbb{Z}$  such that
$$
\gamma_1=\frac{1+mz_1 }{m\rho}\in\Lambda_{0}^1\subset\Lambda \quad {\text{and}}
\quad \gamma_2=\frac{1+mz_2 }{m\rho^{2}}\in\Lambda_{0}^2\subset\Lambda.$$
Obviously, $\gamma_1\neq\gamma_2$. By the orthogonality of $\Lambda$, it follows from (\ref{1.9}) that
\begin{align}\label{3.11.0}
\gamma_1-\gamma_2\in(\Lambda-\Lambda)\setminus\{0\}\subset\mathcal{Z}(\hat{\mu}_{\rho,D})
=\mathcal{Z}(\hat{\mu}_{\rho,\bar{D}})\cup\mathcal{Z}(\hat{\mu}_{\rho,\mathcal{D}_3}).
\end{align}
By using (\ref{2-4}), one may obtain that there exist nonzero integers $a_1, a_2$ and positive integers $j_1, j_2$ such that
\begin{align}\label{3.9.0}
\frac{1+mz_1}{m\rho}-\frac{1+mz_2}{ m\rho^{2}}=\frac{a_1}{\rho^{j_1}2mN}
\end{align}
or
\begin{align}\label{3.10.0}
\frac{1+mz_1}{m\rho}-\frac{1+mz_2}{ m\rho^{2}}=\frac{a_2}{2\rho^{j_2}(1+m\rho-2Nm)}.
\end{align}

Assume that \eqref{3.9.0} holds, and let $u=\rho^{r}$ and $j_1=s+tr$ with $1\leq s\leq r$ and $t\in\mathbb{N}$. At this time, the equation $(\ref{3.9.0})$ can be written as
\begin{align}\label{3.5}
2N\rho^{s-1}u^{t}(1+mz_1)-2N\rho^{s-2}u^{t}(1+mz_2)=a_1.
\end{align}
Since the coefficients of $\rho^{s-1}$ and $\rho^{s-2}$ are nonzero and $a_1\neq 0$,
it follows that the above equation cannot hold because $x^{r}-u\in\mathbb{Q}[x]$ is  the minimal polynomial  of $\rho$. Hence we have $$\gamma_1-\gamma_2\notin \mathcal{Z}(\hat{\mu}_{\rho,\bar{D}}).$$
Combining this with (\ref{3.11.0}), it is known that $\gamma_1-\gamma_2\in\mathcal{Z}(\hat{\mu}_{\rho,\mathcal{D}_3})$, i.e., \eqref{3.10.0} holds.
Let $j_2=s+rt$, where $1\leq s\leq r$ and $t\in\mathbb{N}$. Then the equation \eqref{3.10.0} gives that
$$
2\rho^{s}u^{t}(1+mz_1)m+2\rho^{s-1}u^{t}((1+mz_1)(1-2mN)-m(1+mz_2))-2
\rho^{s-2}u^{t}(1+mz_2)(1-2mN)=a_2m.
$$
Analogous to the analysis of the equation (\ref{3.5}), we can obtain that the above equation does not hold.
This implies that $\gamma_1-\gamma_2\notin\mathcal{Z}(\hat{\mu}_{\rho,\bar{D}})\cup\mathcal{Z}(\hat{\mu}_{\rho,\mathcal{D}_3})$, which contradicts  with \eqref{3.11.0}.
Therefore, we conclude that $\mu_{\rho,D}$ is not a spectral measure.

\end{proof}

\subsection{The step (D)\label{sec3.3}}
Combining with subsections \ref{sec3.1} and \ref{sec3.2}, it is known that $\rho^{-1}$ must be an integer if $\mu_{\rho,D}$ is a spectral measure. In this subsection, we will further prove that $\rho^{-1}$ actually must be divisible by $2mN$.
The following lemma gives necessary conditions for $\mu_{\rho,D}$ to be a spectral measure in the case of $\rho=p^{-1}$ with $p\in\mathbb{N}$ and $p\geq 2$, which will be useful for the subsequent proof.
\begin{lem}\label{lem2-2'}
Let $\rho=p^{-1}$ with $p\in\mathbb{N}$ and $p\geq 2$.
If $\mu_{\rho,D}$ is a spectral measure, then 2, $m$, and $N$ must be factors of $p$.
\end{lem}
\begin{proof}
We will demonstrate that 2, $m$ and $N$ are all factors of $p$ in each of the three cases by decomposing the digit set
$$D=D_m\oplus 2mD_N \oplus (1+m\rho-2Nm)D_2
=\mathcal{D}_1\oplus \mathcal{D}_2\oplus \mathcal{D}_3$$
such that $D$ satisfies the conditions of Remark \ref{re2.3}.
\begin{enumerate}[(a)]
\item $2\mid p$. Assume that $p\in2\mathbb{Z}+1$, then it is easy to see that
$$\mathcal{Z}(\hat{\delta}_{\rho^{2}\mathcal{D}_3})=\frac{p^{2}(2\mathbb{Z}+1)}{2(1+m\rho-2Nm)}\subset\frac{p(2\mathbb{Z}+1)}{2(1+m\rho-2Nm)}=\mathcal{Z}(\hat{\delta}_{\rho\mathcal{D}_3}).$$

\item $m\mid p$.
Let $d=\gcd (m,p)$ with $1\leq d<m$, $m'=\frac{m}{d}$ and $p'=\frac{p}{d}$. Then $\gcd (m',p')=1$ and $\mathcal{D}_1=D_m=D_d\oplus dD_{m'}$. We thus have $$\mathcal{Z}(\hat{\delta}_{\rho^{2}dD_{m'}})=\frac{p^{2}(\mathbb{Z}\setminus m'\mathbb{Z})}{dm'}\subset\frac{p(\mathbb{Z}\setminus m\mathbb{Z})}{m}=\mathcal{Z}(\hat{\delta}_{\rho\mathcal{D}_1}).$$

\item $N\mid p$. Analogous to the case (b), let $d'=\gcd (N,p)$ with $1\leq d'<N$ and $N'=\frac{N}{d'}$. If we decompose $\mathcal{D}_2$ as $\mathcal{D}_2=2mD_N=2mD_{d'}\oplus2md'D_{N'}$,
then  we have
$$\mathcal{Z}(\hat{\delta}_{\rho^{2}2mdD_{N'}})\subset\mathcal{Z}(\hat{\delta}_{\rho\mathcal{D}_2}).$$
\end{enumerate}
By using Remark \ref{re2.3}, we conclude that $\mu_{\rho,D}$ cannot be a spectral measure for all three cases above. This contradiction implies that $p$ necessarily contains 2, $m$ and $N$ as factors.



\end{proof}

Let $m=2^{s_1}m'$ and $N=2^{s_2}N'$ for $s_1,s_2\geq 0$ and $m',N'\in2\mathbb{Z}+1$. To prove that $2mN \mid p$, we first use Lemma \ref{lem2-2'} to show $m'N'\mid p$ under the assumption $2^{s_1+s_2+1}\mid p$.

\begin{pro}\label{pro4.0}
Let $m=2^{s_1}m'$ and $N=2^{s_2}N'$ for $s_1,s_2\geq 0$ and $m', N'\in2\mathbb{Z}+1$.
If $\mu_{\rho,D}$ is a spectral measure with $\rho=p^{-1}(p\in\mathbb{N})$ and $2^{s_1+s_2+1}\mid p$, then $m'N'\mid p$.
\end{pro}

\begin{proof}
By Lemma \ref{lem2-2'}, the assumption that $\mu_{\rho,D}$ is a spectral measure yields that $2\mid p$, $m\mid p$ and $N\mid p$.
If $\gcd{(m',N')}=1$, the desired result naturally follows since $m\mid p$ and $ N \mid p$. Otherwise, if $\gcd{(m',N')}=d>1$, we let $m'=d\bar{m}$ and $N'=d\bar{N}$ with $\gcd{(\bar{m},\bar{N})}=1$. Then we have $d\bar{m}\bar{N}\mid p$. Combining the hypothesis $2^{s_1+s_2+1}\mid p$ and $d\bar{m}\bar{N}\in2\mathbb{Z}+1$, we can set $p=2^{s_1+s_2+1}d\bar{m}\bar{N}p'$ for some $p'\in\mathbb{Z}$. Next, we will prove $d\mid p'$.

Suppose that $d\nmid p'$. Let $a=\gcd(d,p')$ with $1\leq a<d$, $\bar{d}=\frac{d}{a}$ and $\bar{p}=\frac{p'}{a}$, then we have
$m=2^{s_1}\bar{m}a\bar{d}$, $N=2^{s_2}\bar{N}a\bar{d}$ and $p=2^{s_1+s_2+1}\bar{d}\bar{m}\bar{N}a^2\bar{p}$ with $\gcd(\bar{d},\bar{p})=1$. Decompose the digit set
$\bar{D}=\mathcal{D}_1\oplus\mathcal{D}_2=D_m\oplus 2mD_N$ as follows:
\begin{align}\label{3.12}
\bar{D}=(D_{\bar{d}}\oplus \bar{d}D_{2^{s_1}\bar{m}a})\oplus 2m(D_{2^{s_2}\bar{N}a}\oplus 2^{s_2}\bar{N}aD_{\bar{d}} ).
\end{align}
 Since $\gcd(\bar{p}, \bar{d})=1$, it follows that
\begin{align*}
 \mathcal{Z}(\hat{\delta}_{\rho^{2}2m2^{s_2}\bar{N}aD_{\bar{d}}})=\frac{p^{2}(\mathbb{Z}\setminus \bar{d}\mathbb{Z})}{2m2^{s_2}\bar{N}a\bar{d}}
 =\frac{p2^{s_1+s_2+1}\bar{d}\bar{m}\bar{N}a^2\bar{p}(\mathbb{Z}\setminus \bar{d}\mathbb{Z})}{2^{s_1+s_2+1}\bar{m}\bar{N}a^2{\bar{d}}^2}\subset\frac{p(\mathbb{Z}\setminus \bar{d}\mathbb{Z})}{\bar{d}}=\mathcal{Z}(\hat{\delta}_{\rho D_{\bar{d}}}).
\end{align*}
This implies that $\mu_{\rho,D}$ is not a spectral measure  by Remark \ref{re2.3}, which leads to a contradiction. Therefore, we get $d\mid p^{'}$ and complete the proof.
\end{proof}

Below we demonstrate that $2mN\mid p$ is a necessary condition for $\mu_{\rho,D}$ to be a spectral measure when $\rho^{-1}=p\in\mathbb{N}$.
\begin{pro}\label{pro4}
If $\mu_{\rho,D}$ is a spectral measure with $\rho=p^{-1}(p\in\mathbb{N})$, then $2mN\mid p$.
\end{pro}

\begin{proof}
We divide the proof into four cases: \textbf{(i)} $m,N\in2\mathbb{Z}+1$; \textbf{(ii)} $m,N\in2\mathbb{Z}$; \textbf{(iii)} $m\in2\mathbb{Z}+1$ and $N\in2\mathbb{Z}$; \textbf{(iv)} $m\in2\mathbb{Z}$ and $N\in2\mathbb{Z}+1$.
Let $m=2^{s_1}m'$ and $N=2^{s_2}N'$ for $s_1,s_2\geq 0$ and $m',N'\in2\mathbb{Z}+1$. For cases \textbf{(i)}, \textbf{(ii)} and \textbf{(iii)}, we will show that $2^{s_1+s_2+1}\mid p$. Then the conclusion $2mN\mid p$ follows from Proposition \ref{pro4.0} directly.

\textbf{Case (i):} $s_1=s_2=0$. The desired conclusion follows immediately from $2\mid p$.

\textbf{Case (ii):} $s_1,s_2\geq1$. Suppose that $2^{s_1+s_2+1}\nmid p$. Since $m\mid p$ and $N\mid p$, we can assume that $p=2^{s}m^{'}p_{1}$ satisfying $\max\{s_1,s_2\}\leq s\leq s_1+s_2$ and $p_{1}\in2\mathbb{Z}+1$.
Obverse that $\bar{D}$ can be written as
\begin{align*}
\bar{D}=D_2\oplus\cdots\oplus2^{s_1-1}D_2\oplus2^{s_1}D_{m'}\oplus2^{s_1+1}m'D_2\oplus\cdots\oplus2^{s_1+s_2}m'D_2\oplus2^{s_1+s_2+1}m'D_{N'}
\end{align*}
and the range of $s$ in $p$ implies that there must exist $\tilde{s}_1\in \{0,1,\cdots, s_1-1\}$ and $\tilde{s}_2\in \{s_1+1,s_1+2,\cdots, s_1+s_2\}$  such that
$$
\tilde{s}_1+rs=\tilde{s}_2,
$$
where $r=1$ if $s\geq s_1\geq2$ and $r=2$ if $s=s_1=1$.
Through direct computation, we have
$$
\mathcal{Z}(\hat{\delta}_{\rho^{r+1}2^{\tilde{s}_2}m'D_{2}})
=\frac{p^{r+1}(2\mathbb{Z}+1)}{2^{\tilde{s}_2+1}m'}
=\frac{pp^{r}_{1}{m'}^{r-1}(2\mathbb{Z}+1)}{2^{\tilde{s}_1+1}}
\subset\frac{p(2\mathbb{Z}+1)}{2^{\tilde{s}_1+1}}=\mathcal{Z}(\hat{\delta}_{\rho 2^{\tilde{s}_1}D_{2}}).
$$
Combining with Remark \ref{re2.3}, it is known that $\mu_{\rho,D}$ is not a spectral measure, which contradicts with our assumption. Hence we have $2^{s_1+s_2+1}\mid p$.

\textbf{Case (iii):} $s_1=0,s_2\geq 1$. Let $m=d\bar{m}$ and $N=2^{s_2}d\bar{N}$ with $d=\gcd{(m,N)}\geq1$ and $\gcd{(\bar{m},\bar{N})}=1$. Since both $m$ and $N$ divide $p$, we may express $p$ in the form: $p=2^{s_2}d\bar{m}\bar{N}\tilde{p}$ for some $\tilde{p}\in\mathbb{Z}$.
Now we prove $\tilde{p}\in2\mathbb{Z}$ by contradiction.
Assume the opposite $\tilde{p}\in2\mathbb{Z}+1$. Recall that
$$
D=\mathcal{D}_1\oplus \mathcal{D}_2\oplus \mathcal{D}_3
=D_m\oplus 2mD_N \oplus \frac{m+p(1-2mN)}{p}D_2.
$$
By noting that $m+p(1-2mN)\in 2\mathbb{Z}+1$ and
$$
2mD_N=2m(D_{2^{s_{2}-1}\bar{N}}\oplus2^{s_{2}-1}\bar{N}D_2\oplus2^{s_{2}}\bar{N}D_d),
$$
one may obtain
$$
\mathcal{Z}(\hat{\delta}_{\rho^{3}2^{s_2}m\bar{N}D_2})
=\frac{p^{2}\tilde{p}(2\mathbb{Z}+1)}{2}\subset\frac{p^{2}(2\mathbb{Z}+1)}{2}
\subset\frac{p^{2}(2\mathbb{Z}+1)}{2(m+p(1-2mN))}=\mathcal{Z}(\hat{\delta}_{\rho \mathcal{D}_3}).
$$
Thus, we have $\tilde{p}\in2\mathbb{Z}$ and $2^{s_2+1}\mid p$ by using Remark \ref{re2.3} again.

So far, we have shown that $2mN\mid p$ for three cases except for the case \textbf{(iv)}. Next, let us consider the last case ($s:=s_1\geq1, s_2=0$).

\textbf{Case (iv):} Recalling that $m=2^{s}m'$, we first show that $2^{s}m'N\mid p$. When $\gcd(m',N)=1$, this result clearly holds since both $m\mid p$ and $N\mid p$. When $d=\gcd(m',N)>1$, we decompose $m'=\bar{m}d$ and $N=\bar{N}d$ for some $\bar{m},\bar{N}\in\mathbb{Z}$ with $\gcd(\bar{m},\bar{N})=1$. The divisibility conditions $m\mid p$ and $N\mid p$ imply that $2^{s}\bar{m}\bar{N}d\mid p$, so we can write $p=2^{s}\bar{m}\bar{N}dp'$ for some $p'\in\mathbb{Z}$. To complete the proof, it suffices to show $d\mid p'$.

Without loss of generality, let $d'=\gcd(d,p')\geq1$. We factorize $d=d'\bar{d}$ and $p'=d'\bar{p}$ with $\gcd(\bar{d},\bar{p})=1$, yielding the descomposition $p=2^{s}\bar{m}\bar{N}d'\bar{d}\bar{p}$. When the digit set $\bar{D}$ is expressed as
$$\bar{D}=(D_{2}\oplus2D_{\bar{d}}\oplus 2\bar{d}D_{d'\bar{m}2^{s-1}})\oplus 2mD_{\bar{N}}\oplus2m \bar{N}(D_{\bar{d}}\oplus \bar{d}D_{d'}),$$
we can get
$$
\mathcal{Z}(\hat{\delta}_{\rho^{2}2m\bar{N}D_{\bar{d}}})=\frac{p^{2}(\mathbb{Z}\setminus \bar{d}\mathbb{Z})}{2m\bar{N}\bar{d}}=\frac{p\bar{p}(\mathbb{Z}\setminus \bar{d}\mathbb{Z})}{2\bar{d}}\subset\frac{p(\mathbb{Z}\setminus \bar{d}\mathbb{Z})}{2\bar{d}}=\mathcal{Z}(\hat{\delta}_{\rho 2D_{\bar{d}}}).
$$
By using  Remark \ref{re2.3}, it follows that $d\mid p'$. Hence we have $2^{s}m'N\mid p$.

Finally, we show $2^{s+1}m'N\mid p$. Let $p=m\tilde{p}=2^{s}m'\tilde{p}$ for some $\tilde{p}\in\mathbb{Z}$. Since $2^{s}m'N\mid p$, it follows that $N\mid \tilde{p}$. We now prove that $\tilde{p}$ must be even. Suppose, to the contrary, that $\tilde{p}\in2\mathbb{Z}+1$. Then it is easy to see that $1+\tilde{p}(1-2mN)\in 2\mathbb{Z}$. Write $1+\tilde{p}(1-2mN):=2^{sl+a}t$, where $t\in2\mathbb{Z}+1$ and $sl+a\geq 1$ with $l\in \mathbb{N}$ and $a\in\{0,1,\cdots,s-1\}$. Note that $\gcd{(t,N)}=1$ since $N\mid\tilde{p}$.
Combining (\ref{1.11}) and $\rho=p^{-1}$, we decompose $D$ as follows:
\begin{align*}
D&=D_m\oplus 2mD_N \oplus (m+p(1-2mN))p^{-1}D_2\nonumber\\
&=D_m\oplus 2mD_N \oplus 2^{sl+a}t\tilde{p}^{-1}D_2\nonumber\\
&=D_2\oplus\cdots\oplus 2^{a}D_2\oplus\cdots\oplus2^{s-1}D_2\oplus 2^{s}D_{m^{'}}\oplus 2mD_N \oplus 2^{sl+a}t\tilde{p}^{-1}D_2.
\end{align*}
We will derive a contradiction by considering two separate situations: $a=0$ and $a\neq0$.

\ding{172} $a=0$: The fact $a=0$ forces $l\geq 1$ since $sl+a\geq1$. For the case $l\geq2$, consider the factorization
$\mu_{\rho,D}=\omega_1*\omega_2$, where $$\omega_1=*_{j=1}^{l-1}\delta_{\rho^{j}2^{sl+a}t\tilde{p}^{-1}D_2} \ \ {\rm{and}} \ \  \omega_2=(*_{j=l}^{\infty}\delta_{\rho^{j}2^{sl+a}t\tilde{p}^{-1}D_2})*(*_{j=1}^{\infty}\delta_{\rho^{j}D_m})*(*_{j=1}^{\infty}\delta_{\rho^{j}2mD_N}).$$
By using $\gcd{(t,N)}=1$, it can be verified that the assumption ($\star$) with respect to ($\omega_1,\omega_2$) holds. If $\mu_{\rho,D}$ is a spectral measure, then $\omega_{2}$ is a spectral measure by Theorem \ref{th4}. For the case $l=1$, we just take $\omega_{2}=\mu_{\rho,D}$. Let $\omega_{2}$ be written as a Moran measure in terms of the following way.
\begin{align*}
\omega_{2}=&\delta_{p^{-1}2m D_N}*\delta_{p^{-l}2^{sl}t\tilde{p}^{-1} D_2}*\delta_{p^{-1}2^{s}D_{m'}}*\delta_{p^{-1}2^{s-1}D_{2}}*\cdots*\delta_{p^{-1}D_{2}}*\delta_{p^{-(l+1)}2^{sl}t\tilde{p}^{-1} D_2}\\
&*\delta_{p^{-2}2m D_N}*\delta_{p^{-2}2^{s}D_{m'}}*\delta_{p^{-2}2^{s-1}D_{2}}*\cdots*\delta_{p^{-2}D_{2}}*\delta_{p^{-(l+2)}2^{sl}t\tilde{p}^{-1} D_2}*\delta_{p^{-3}2m D_N}*\delta_{p^{-3}2^{s}D_{m'}}*\cdots\\
=&\delta_{(\frac{p}{2m})^{-1}D_{N}}*\delta_{(\frac{p}{2m})^{-1}(\frac{2(m'\tilde{p})^{l}}{t})^{-1}D_2}*\delta_{(\frac{p}{2m})^{-1}(\frac{2(m'\tilde{p})^{l}}{t})^{-1}(\frac{mt}{2^{s}(m'\tilde{p})^{l}})^{-1}D_{m'}}*\delta_{(\frac{p}{2m})^{-1}(\frac{2(m'\tilde{p})^{l}}{t})^{-1}(\frac{mt}{2^{s}(m'\tilde{p})^{l}})^{-1}2^{-1}D_2}*\cdots\\
&*\delta_{(\frac{p}{2m})^{-1}(\frac{2(m'\tilde{p})^{l}}{t})^{-1}(\frac{mt}{2^{s}(m'\tilde{p})^{l}})^{-1}\underbrace{2^{-1}\cdots 2^{-1}}_{s}D_2}*\delta_{(\frac{p}{2m})^{-1}(\frac{2(m'\tilde{p})^{l}}{t})^{-1}(\frac{mt}{2^{s}(m'\tilde{p})^{l}})^{-1}\underbrace{2^{-1}\cdots 2^{-1}}_{s}(\frac{m'^{l}\tilde{p}^{l+1}}{t})^{-1}D_2}\\
&*\delta_{(\frac{p}{2m})^{-1}(\frac{2(m'\tilde{p})^{l}}{t})^{-1}(\frac{mt}{2^{s}(m'\tilde{p})^{l}})^{-1}\underbrace{2^{-1}\cdots 2^{-1}}_{s}(\frac{m'^{l}\tilde{p}^{l+1}}{t})^{-1}(\frac{t}{2(m'\tilde{p})^{l}})^{-1}D_N}*\cdots\\
:=&\delta_{b_{1}^{-1}R_1}\ast\delta_{b_{1}^{-1}b_{2}^{-1}R_{2}}\ast\delta_{b_{1}^{-1}b_{2}^{-1}b_{3}^{-1}R_{3}}\ast\cdots.
\end{align*}
Let $b_k=\frac{l_k}{t_k}$ and $R_k=\{0,1,\cdots,\gamma_k-1\}$ for $k\geq 1$ as in Lemma \ref{pro7}, then
it is clear that $\{t_k\}_{k=1}^{\infty}\subset\{2m,t,2^{s}m'^{l}\tilde{p}^{l},1,2m'^{l}\tilde{p}^{l}\}$ and $\{\gamma_k\}_{k=1}^{\infty}\subset\{N,m',2\}$  are both bounded.
Moreover, observe that
$$
\delta_{p^{-1}D_2}
=\delta_{b_1^{-1}\cdots b_{s+3}^{-1}R_{s+3}}=\delta_{(\frac{p}{2m})^{-1}(\frac{2(m'\tilde{p})^{l}}{t})^{-1}(\frac{mt}{2^{s}(m'\tilde{p})^{l}})^{-1}\underbrace{2^{-1}\cdots 2^{-1}}_{s}D_2}
$$
and
$$
\delta_{p^{-(l+1)}2^{sl}t\tilde{p}^{-1} D_2 }
=\delta_{b_1^{-1}\cdots b_{s+3}^{-1}(\frac{m'^{l}\tilde{p}^{l+1}}{t})^{-1}R_{s+4}}=\delta_{(\frac{p}{2m})^{-1}(\frac{2(m'\tilde{p})^{l}}{t})^{-1}(\frac{mt}{2^{s}(m'\tilde{p})^{l}})^{-1}\underbrace{2^{-1}\cdots 2^{-1}}_{s}(\frac{m'^{l}\tilde{p}^{l+1}}{t})^{-1}D_2}.
$$
Using Lemma \ref{pro7} (iii), by taking $i=s+3$, we have $2\mid m'^{l}\tilde{p}^{l+1}$ since $2\nmid t$. However, this is impossible because $m^{'l}\tilde{p}^{l+1}\in2\mathbb{Z}+1$.

\ding{173} $a>0$: It follows that $l\geq 0$ by $sl+a\geq1$. If $l\geq1$, we express $\mu_{\rho,D}$ as the convolution product $\omega_1*\omega_2$, where $$\omega_1=*_{j=1}^{l}\delta_{\rho^{j}2^{sl+a}t\tilde{p}^{-1}D_2} \ \ {\rm{and}} \ \  \omega_2=(*_{j=l+1}^{\infty}\delta_{\rho^{j}2^{sl+a}t\tilde{p}^{-1}D_2})*(*_{j=1}^{\infty}\delta_{\rho^{j}D_m})*(*_{j=1}^{\infty}\delta_{\rho^{j}2mD_N}).$$
Following an argument analogous to the case $a=0$,
we conclude that $\omega_{2}$ is a spectral measure. If $l=0$, we simply take $\omega_{2}=\mu_{\rho,D}$. Write $\omega_{2}$ as a Moran measure in the following way:
\begin{align*}
\omega_{2}=&\delta_{p^{-1}2m D_N}*\delta_{p^{-1}2^{s}D_{m'}}*\delta_{p^{-1}2^{s-1}D_{2}}*\cdots*\delta_{p^{-1}2^{a+1}D_{2}}*\delta_{p^{-(l+1)}2^{sl+a}t\tilde{p}^{-1} D_2}*\delta_{p^{-1}2^{a}D_{2}}*\cdots*\delta_{p^{-1}D_{2}}\\
&*\delta_{p^{-2}2m D_N}*\delta_{p^{-2}2^{s}D_{m'}}*\delta_{p^{-2}2^{s-1}D_{2}}*\cdots*\delta_{p^{-(l+2)}2^{sl+a}t\tilde{p}^{-1} D_2}*\delta_{p^{-2}2^{a}D_{2}}*\cdots*\delta_{p^{-2}D_{2}}*\cdots\\
=&\delta_{(\frac{p}{2m})^{-1} D_N}*\delta_{(\frac{p}{2m})^{-1}(\frac{m}{2^{s-1}})^{-1} D_{m'}}*\delta_{(\frac{p}{2m})^{-1}(\frac{m}{2^{s-1}})^{-1}2^{-1}D_2}*\cdots*\delta_{(\frac{p}{2m})^{-1}(\frac{m}{2^{s-1}})^{-1}\underbrace{2^{-1}\cdots2^{-1}}_{s-a-1}D_2}\\
&*\delta_{(\frac{p}{2m})^{-1}(\frac{m}{2^{s-1}})^{-1}\underbrace{2^{-1}\cdots2^{-1}}_{s-a-1}(\frac{2m'^{l}\tilde{p}^{l+1}}{t})^{-1}D_2}*\delta_{(\frac{p}{2m})^{-1}(\frac{m}{2^{s-1}})^{-1}\underbrace{2^{-1}\cdots2^{-1}}_{s-a-1}(\frac{2m'^{l}\tilde{p}^{l+1}}{t})^{-1}(\frac{t}{m'^{l}\tilde{p}^{l+1}})^{-1}D_2}*\cdots\\
:=&\delta_{b_{1}^{-1}R_{1}}\ast\delta_{b_{1}^{-1}b_{2}^{-1}R_{2}}\ast\delta_{b_{1}^{-1}b_{2}^{-1}b_{3}^{-1}R_{3}}\ast\cdots,
\end{align*}
where
$$
\delta_{p^{-(l+1)}2^{sl+a}t\tilde{p}^{-1} D_2}
=\delta_{b_1^{-1}\cdots b_{s-a+2}^{-1}R_{s-a+2}}=\delta_{(\frac{p}{2m})^{-1}(\frac{m}{2^{s-1}})^{-1}\underbrace{2^{-1}\cdots2^{-1}}_{s-a-1}(\frac{2m'^{l}\tilde{p}^{l+1}}{t})^{-1}D_2}
$$
and
$$
\delta_{p^{-1}2^{a}D_2}=\delta_{b_1^{-1}\cdots b_{s-a+2}^{-1}(\frac{t}{m'^{l}\tilde{p}^{l+1}})^{-1}R_{s-a+3}}=\delta_{(\frac{p}{2m})^{-1}(\frac{m}{2^{s-1}})^{-1}\underbrace{2^{-1}\cdots2^{-1}}_{s-a-1}(\frac{2m'^{l}\tilde{p}^{l+1}}{t})^{-1}(\frac{t}{m'^{l}\tilde{p}^{l+1}})^{-1}D_2}.
$$
One can easily verify that Lemma \ref{pro7} applies to this case as well. Using Lemma \ref{pro7} (iii) for $i=s-a+2$, we can get $2\mid t$,  a contradiction. Therefore, we have $\tilde{p}\in2\mathbb{Z}$ and finish the proof.

\end{proof}
\subsection{Proof of Theorem \ref{thm1.2}\label{sec3.4}}
Having established the direction $(i)\Rightarrow (ii)$ of Theorem \ref{thm1.2} in the previous subsections, we next give a complete proof of this theorem.

\begin{proof}[Proof of Theorem \ref{thm1.2}]
$(i)\Rightarrow (ii):$ Combining  Propositions \ref{pro1}, \ref{pro2}, \ref{pro3'} and \ref{pro4}, we can show this direction immediately.

$(ii)\Rightarrow (iii):$ Assume that $p=2mNp'$ for some $p'\in\mathbb{Z}$.
Recall that
$$D=\{0,1+m\rho-2Nm\}\oplus\{0,1,\cdots,m-1\}\oplus2m\{0,1\cdots,N-1\}.$$
A direct calculation gives
\begin{align*}
pD&=\{0,(1-2mN)p+m\}\oplus p\{0,1,\cdots,m-1\}\oplus 2mp\{0,1\cdots,N-1\}\nonumber\\
&:=\mathcal{E}_0\oplus p\mathcal{E}_1 \oplus p\mathcal{E}_2,
\end{align*}
where $\mathcal{E}_0=\{0,(1-2mN)p+m\}$, $\mathcal{E}_1=D_m$ and $\mathcal{E}_2=2mD_{N}$.
Choose $L_0=\{0,Np'\}$, $L_1=2Np'\{0,1,\cdots,m-1\}$ and $L_2=p'\{0,1,\cdots,N-1\}$. By using Lemma \ref{lem2-6}, it is easy to check that $(p,\mathcal{E}_i,L_i)$ for $i\in\{0,1,2\}$, $(p,\mathcal{E}_0\oplus \mathcal{E}_1,L_0\oplus L_1)$, $(p,\mathcal{E}_1\oplus \mathcal{E}_2,L_1\oplus L_2)$ and $(p,\mathcal{E}_0\oplus \mathcal{E}_1\oplus \mathcal{E}_2,L_0\oplus L_1\oplus L_2)$ are all Hadamard triples. Therefore, write $L:=L_0\oplus L_1\oplus L_2$, then $(p,pD,L)$ is a 2-stage product-form Hadamard triple.

$(iii)\Rightarrow (i):$ From Theorem \ref{th5}, we conclude that $\mu_{\rho,pD}$ is a spectral measure. Hence $\mu_{\rho,D}$ is a spectral measure by Lemma \ref{lem4}.
\end{proof}

\section{Nonspectrality of  self-similar measures  \label{sec4}}
For the self-similar measure $\nu_{\rho, D_s}^{1}$ generated by the IFS $\{\tau_d(\cdot)=(-1)^{d}\rho(\cdot+d)\}_{d\in D_s}$,
Wu \cite{Wu2024} has characterised the spectrality of $\nu_{\rho, D_s}^{1}$ when $s\in 2\mathbb{N}$ (Theorem \ref{thm1.0}) by using an infinite product of a function matrix instead of the infinite product of a mask polynomial to express $\hat{\nu}_{\rho, D_s}^{1}$. Unfortunately, this method does not work for the situation where $s$ is an odd number. In this section, for the case $s\in\mathbb{N}$ with $s\geq2$, the number of orthogonal exponentials of $L^2(\nu_{\rho, D_s}^{1})$ is estimated under the condition that $\rho^{-1}\in\mathbb{N}$ and $\gcd{(\rho^{-1},s)}=1$. Especially for the case $s\in2\mathbb{N}+1$, we achieve the goal by exploiting some properties of the Fourier transform of the measure to give the possible range of zeros of $\hat{\nu}_{\rho, D_s}^{1}$.  Theorem \ref{thm1.4} follows directly from the following proposition.

\begin{pro}\label{thm4.1}
Let $\nu_{\rho,D_s}^{1}:=\nu'$ be a self-similar measure generated by (\ref{1.12}), where $0<\rho<1$ and $s\geq 2$ is a positive integer. Then we have the following two conclusions.
\begin{enumerate}[(i)]
\item For $s\in 2\mathbb{N}+1$, we have $\mathcal{Z}(\hat{\nu}') \subset \left(\bigcup_{k=1}^{\infty}\frac{(2\mathbb{Z}+1)\backslash s(2\mathbb{Z}+1)}{2\rho^{k}s}\right)\cup \left(\bigcup_{k=1}^{\infty}\frac{\mathbb{Z}\backslash s\mathbb{Z}}{\rho^{k}s}\right)$.

\item If $\rho^{-1}=p\in\mathbb{N}$ and $\gcd{(p,s)}=1$, then $L^2(\nu')$ contains at most $s$ mutually orthogonal exponential functions.

\end{enumerate}
\end{pro}

\begin{proof}
We first prove $(i)$. If $\mathcal{Z}(\hat{\nu}')=\emptyset$, the desired result stands naturally. Below we need only consider the case $\mathcal{Z}(\hat{\nu}')\neq\emptyset$. The assumption $s\in 2\mathbb{N}+1$ yields that $s-1=2N$ for some $N\in \mathbb{N}$. Combining (\ref{1.2}) and (\ref{1.12}), the Fourier transform of $\nu'$ can be expressed as follows:
\begin{align}\label{4.3}
\hat{\nu}'(t)&=\frac{1}{2N+1}\sum_{d=0}^{2N}e^{2\pi i (-1)^{d}\rho td}\hat{\nu}'((-1)^{d}\rho t)\nonumber\\
&=\frac{1}{2N+1}\left(\left(\sum_{j=0}^{N}e^{2\pi i 2j\rho t}\right)\hat{\nu}'(\rho t)+\left(\sum_{j=1}^{N}e^{-2\pi i (2j-1)\rho t}\right)\hat{\nu}'(-\rho t)\right).
\end{align}
Notice that $\hat{\nu}'(\rho^{n}t)$ approaches $\hat{\nu}'(0)=1$ when $n$ is sufficiently large.  Then for any $t\in\mathcal{Z}(\hat{\nu}')$, it follows from the continuity of $\hat{\nu}'$ that there exists a smallest positive integer  $n_t\geq 1$ such that $\hat{\nu}'(\rho^{n_t}t)\neq0$ but $\hat{\nu}'(\rho^{i}t)=0$ for all $i=0,1,\cdots,n_t-1$.
Using $\hat{\nu}'(\rho^{n_t-1}t)=0$ and (\ref{4.3}), it is clear that $2\rho^{n_t} t\notin\mathbb{Z}$. 
Otherwise, we can get $$0=\hat{\nu}'(\rho^{n_t-1}t)=(N+1)\hat{\nu}'(\rho^{n_t} t)\pm N\hat{\nu}'(-\rho^{n_t} t), $$which is impossible since $1=|\frac{\hat{\nu}'(\rho^{n_t} t)}{\overline{\hat{\nu}'(\rho^{n_t} t)}}|=|\frac{\hat{\nu}'(\rho^{n_t} t)}{\hat{\nu}'(-\rho^{n_t} t)}|=\frac{N}{N+1}\neq 1$, where $\overline{\hat{\nu}'}$ denotes the conjugate function of $\hat{\nu}'$.
Hence it follows from $\hat{\nu}'(\rho^{n_t-1}t)=0$ and (\ref{4.3}) that
 $$\frac{\hat{\nu}'(-\rho^{n_t} t)}{\hat{\nu}'(\rho^{n_t} t)}=-\frac{\sum_{j=0}^{N}e^{2\pi i 2j\rho^{n_t} t}}{\sum_{j=1}^{N}e^{-2\pi i (2j-1)\rho^{n_t} t}}=
 -e^{2\pi i \rho^{n_t} t}\frac{\frac{1-e^{2\pi i(2N+2)\rho^{n_t} t}}{1-e^{2\pi i 2\rho^{n_t} t}}}{\frac{1-e^{-2\pi i 2N\rho^{n_t} t}}{1-e^{-2\pi i 2\rho^{n_t} t}}}
 =-e^{2\pi i 2N\rho^{n_t} t}\frac{\sin(2\pi (N+1)\rho^{n_t} t)}{\sin(2\pi N\rho^{n_t} t)}.$$
Based on the fact that $|\hat{\nu}'(-\rho^{n_t} t)|=|\hat{\nu}'(\rho^{n_t} t)|$, the above equation forces $$\sin(2\pi (N+1)\rho^{n_t} t)=\pm\sin(2\pi N\rho^{n_t} t).$$

\textbf{Case I :} Suppose that $\sin(2\pi (N+1)\rho^{n_t} t)=\sin(2\pi N\rho^{n_t} t)$.
A direct calculation gives that
$$\rho^{n_t} t\in\mathbb{Z}\cup\frac{2\mathbb{Z}+1}{2s}.$$
It follows from $2\rho^{n_t} t\notin\mathbb{Z}$ that $\rho^{n_t} t\in\frac{(2\mathbb{Z}+1)\backslash s(2\mathbb{Z}+1)}{2s}$.

Note that $t$ is any zero of $\hat{\nu}'$ and $n_t\geq 1$ is the smallest positive integer  such that $\hat{\nu}'(\rho^{n_t}t)\neq0$ but $\hat{\nu}'(\rho^{i}t)=0$ for all $0\leq i\leq n_t-1$. Therefore, we can get
$$\mathcal{Z}(\hat{\nu}')\subset\bigcup_{k=1}^{\infty}\rho^{-k}\frac{(2\mathbb{Z}+1)\backslash s(2\mathbb{Z}+1)}{2s}.$$

\textbf{Case II :} Suppose that  $\sin(2\pi (N+1)\rho^{n_t} t)=-\sin(2\pi N\rho^{n_t} t)$.  One may obtain that
$\rho^{n_t} t\in\frac{\mathbb{Z}}{s}\cup\frac{2\mathbb{Z}+1}{2}$ by a simple calculation. Since $2\rho^{n_t} t\notin\mathbb{Z}$, it follows that $\rho^{n_t} t\in\frac{\mathbb{Z}\backslash s\mathbb{Z}}{s}$.
A similar analysis as in the previous case yields
$$
\mathcal{Z}(\hat{\nu}')\subset\bigcup_{k=1}^{\infty}\rho^{-k}\mathcal{Z}(m_D)=\bigcup_{k=1}^{\infty}\rho^{-k}\frac{\mathbb{Z}\backslash s\mathbb{Z}}{s}.
$$
Consequently, we have completed the proof of $(i)$.

Next, we prove $(ii)$. Under the assumption that $\rho^{-1}=p\in\mathbb{N}$ with $\gcd (p,s)=1$, part (i) readily implies that
\begin{align*}
\mathcal{Z}(\hat{\nu}')\subset\frac{(2\mathbb{Z}+1)\backslash s(2\mathbb{Z}+1)}{2s}\cup\frac{\mathbb{Z}\backslash s\mathbb{Z}}{s}\subset \frac{\mathbb{Z}\backslash s\mathbb{Z}}{2s}
\end{align*}
provided that $s\in 2\mathbb{N}+1$. On the other hand, if $s=2N$ for some $N\in\mathbb{N}$, then we have $\mathcal{Z}(\hat{\nu}')=\mathcal{Z}(\hat{\mu}_{\rho,D^{'}})$ by Proposition \ref{lem3}, where $D^{'}=2D_N\oplus(1+\rho-2N)D_2$ and $\mu_{\rho,D^{'}}$ is a canonical self-similar measure generated from (\ref{1.6}).
Since $\rho^{-1}=p\in\mathbb{N}$ with $\gcd (p,2N)=1$ and $p(1-2N)+1=:Q\in 2\mathbb{Z}$, one may obtain that
$$\mathcal{Z}(\hat{\nu}')
=\mathcal{Z}(\hat{\mu}_{\rho,D^{'}})=\bigcup_{j=1}^{\infty}p^{j}\left(\frac{\mathbb{Z}\setminus N\mathbb{Z}}{2N}\cup\frac{p(\mathbb{Z}\setminus 2\mathbb{Z})}{2p(1-2N)+2}\right)\subset\frac{\mathbb{Z}\setminus 2N\mathbb{Z}}{2NQ}=\frac{\mathbb{Z}\setminus s\mathbb{Z}}{sQ}.$$

Assume that $0\in\Lambda$ is an orthogonal set of $\nu'$. We claim that $\#\Lambda\leq s$. If $s\in 2\mathbb{N}+1$ and $\#\Lambda\geq s+1$, then there must exist $\lambda_1=\frac{l_1}{2s}\neq\lambda_2=\frac{l_2}{2s}\in\Lambda$ such that $l_1=l_2 \pmod{s\mathbb{Z}}$. Thus $\lambda_1-\lambda_2\in \frac{\mathbb{Z}}{2}\not\subset\mathcal{Z}(\hat{\nu}')$. Similarly, if
$s\in2\mathbb{N}$, then there must exist $\lambda_1, \lambda_2$ such that $\lambda_1-\lambda_2\in \frac{\mathbb{Z}}{Q}\not\subset\mathcal{Z}(\hat{\nu}')$. These contradicts with (\ref{1.9}). Therefore, one may obtain  $\#\Lambda\leq s$, which implies that $L^2(\nu')$ contains at most $s$ mutually orthogonal exponential functions.
\end{proof}

Proposition \ref{thm4.1} provides a possible range for $\mathcal{Z}(\hat{\nu}')$
under the assumption that $\mathcal{Z}(\hat{\nu}')\neq\emptyset$. However, both the existence and explicit form of $\mathcal{Z}(\hat{\nu}')$ remain undetermined with current methods. Furthermore, for the measure $\nu^{m}_{\rho,D_{(2N+1)m}}$, unlike the analysis for the measure $\nu_{\rho,D_{2Nm}}^{m}$, our methods cannot establish the relationship between $\hat{\nu}^{m}_{\rho,D_{(2N+1)m}}(t)$ and $\hat{\nu}^{m}_{\rho,D_{(2N+1)m}}(-t)$. Consequently, we are unable to fully characterize the Fourier transform of $\nu^{m}_{\rho,D_{(2N+1)m}}$. Therefore, we propose the following open problem:
\begin{question*}
For the self-similar measure $\nu_{\rho, D_{(2N+1)m}}^{m}$ with $N\in\mathbb{N}$ generated by the IFS in (\ref{1.3}), what are the zeros of its Fourier transform? Moreover, what is the sufficient and necessary condition for $\nu_{\rho, D_{(2N+1)m}}^{m}$ to be a spectral measure?
\end{question*}
We conclude this section by establishing the Fourier transform representation for a class of self-similar measures $\mathbf{v}$ with alternative contraction ratios. This is achieved by utilizing the fundamental symmetry property $\hat{\mathbf{v}}(t)=\hat{\mathbf{v}}(-t)$ for any $t\in\mathbb{R}$. Although this example has been investigated in \cite{Wu-Liu2022}, the method used here provides a distinct perspective.

\begin{ex}
Let $\nu_{\rho,\tilde{D}_{2n+1}}^{1}$ be a self-similar measure generated by the IFS in (\ref{1.3}), where $0<\rho<1$ and $\tilde{D}_{2n+1}=\{-n,-(n-1),\cdots,-1,0,1,\cdots,n\}$. Then $\hat{\nu}_{\rho,\tilde{D}_{2n+1}}^{1}(t)=e^{\frac{-2\pi int\rho}{1-\rho}}\hat{\mu}_{\rho,D_{2n+1}}(t)$, where $\mu_{\rho,D_{2n+1}}$ is a canonical self-similar measure defined by (\ref{1.6}).
\end{ex}
\begin{proof} 
Write $\mathbf{v}:=\nu_{\rho,\tilde{D}_{2n+1}}^{1}$ for convenience. For any $t\in\mathbb{R}$, combining (\ref{1.2}) and (\ref{1.3}) yields the following expression for $\hat{\mathbf{v}}(t)$:
\begin{align}\label{4.1}
\hat{\mathbf{v}}(t)=\frac{1}{2n+1}\left(\sum_{j=1}^{n}\hat{\mathbf{v}}((-1)^{j}\rho t)(e^{2\pi i j\rho t}+e^{-2\pi i j\rho t})+\hat{\mathbf{v}}(\rho t)\right).
\end{align}
Through direct calculation, we have
$$\hat{\mathbf{v}}(-t)=\frac{1}{2n+1}\left(\sum_{j=1}^{n}\hat{\mathbf{v}}((-1)^{j+1}\rho t)(e^{-2\pi i j\rho t}+e^{2\pi i j\rho t})+\hat{\mathbf{v}}(-\rho t)\right)$$
and
\begin{equation}\label{4.5}
|\hat{\mathbf{v}}(t)-\hat{\mathbf{v}}(-t)|\leq\frac{|\hat{\mathbf{v}}(\rho t)-\hat{\mathbf{v}}(-\rho t)|}{2n+1}\left(\sum_{j=1}^{n}|e^{-2\pi i j\rho t}+e^{2\pi i j\rho t}|+1\right)
\leq|\hat{\mathbf{v}}(\rho t)-\hat{\mathbf{v}}(-\rho t)|.
\end{equation}
Observe that $\hat{\mathbf{v}}(\pm\rho^{m} t)$ converges to $\hat{\mathbf{v}}(0)=1$ as $m$ becomes sufficiently large. Thus, for any $\varepsilon>0$, there exists $\delta>0$ such that for $|\rho^{m}t|<\delta$ we have $|\hat{\mathbf{v}}(\pm\rho^{m} t)-1|<\varepsilon$.
Selecting a large enough integer $N$ such that $|\rho^{m}t|<\delta$ for $m>N$, we obtain
$|\hat{\mathbf{v}}(\rho^{m} t)-\hat{\mathbf{v}}(-\rho^{m} t)|<2\varepsilon$. By iterating (\ref{4.5}) $m~(m>N)$ times, we arrive at
$$|\hat{\mathbf{v}}(t)-\hat{\mathbf{v}}(-t)|\leq|\hat{\mathbf{v}}(\rho t)-\hat{\mathbf{v}}(-\rho t)|\leq\cdots\leq|\hat{\mathbf{v}}(\rho^{m} t)-\hat{\mathbf{v}}(-\rho^{m} t)|<2\varepsilon.$$
Since $\varepsilon$ is arbitrary, we conclude that
$\hat{\mathbf{v}}(t)=\hat{\mathbf{v}}(-t)$. Then it follows from $(\ref{4.1})$ that
\begin{align*}
\hat{\mathbf{v}}(t)&=\frac{1}{2n+1}\sum_{j=-n}^{n}e^{2\pi i j\rho t}\hat{\mathbf{v}}(\rho t)=e^{-2\pi in\rho t}\frac{1}{2n+1}\sum_{j=0}^{2n}e^{2\pi i j\rho t}\hat{\mathbf{v}}(\rho t)\\
&=\prod_{k=1}^{\infty}e^{-2\pi in\rho^{k}t}\prod_{k=1}^{\infty}m_{D_{2n+1}}(\rho^{k}t)
=e^{\frac{-2\pi int\rho}{1-\rho}}\hat{\mu}_{\rho,D_{2n+1}}(t),
\end{align*}
where $\mu_{\rho,D_{2n+1}}$ is a canonical self-similar measure defined by (\ref{1.6}).
\end{proof}
Combining with Lemma \ref{pro6}, we can immediately establish the equivalence of spectrality between $\nu_{\rho,\tilde{D}_{2n+1}}^{1}$ and $\mu_{\rho,D_{2n+1}}$. Thus, by \cite[Theorem 1.1]{Dai-He-Lau2014}, we conclude that $\nu_{\rho,\tilde{D}_{2n+1}}^{1}$ is a spectral measure if and only if $\rho^{-1}\in (2n+1)\mathbb{N}^{+}$.
As can be seen from this example, the advantage of connecting the Fourier transform of $\nu_{\rho,\tilde{D}_{2n+1}}^{1}$ to that of $\mu_{\rho,D_{2n+1}}$ lies in enabling characterization of $\nu_{\rho,\tilde{D}_{2n+1}}^{1}$'s spectrality by leveraging established results about the canonical self-similar measure $\mu_{\rho,D_{2n+1}}$. 

\end{document}